\documentclass[12pt]{article}

\usepackage{multicol}
\usepackage{amssymb,amsmath,amsthm}
\usepackage[mathscr]{eucal}
\usepackage[spanish,english]{babel}   
\usepackage[latin1]{inputenc}
\usepackage[dvips]{graphicx}
\usepackage{multirow}                
\usepackage{array}
\usepackage{graphpap}
\usepackage{layout}

\theoremstyle{plain}
\newtheorem{theorem}{Theorem}[section]
\newtheorem{lemma}[theorem]{Lemma}
\newtheorem{corol}[theorem]{Corollary}
\newtheorem{prop}[theorem]{Proposition}
\newtheorem{definition}[theorem]{Definition}
\theoremstyle{remark}
\newtheorem{example}[theorem]{Example}
\newtheorem{remark}{Remark}

\newtheorem*{theoremA}{Theorem A}
\newtheorem*{theoremB}{Theorem B}
\newtheorem*{theoremC1}{Theorem C1}
\newtheorem*{theoremC2}{Theorem C2}
\newtheorem*{theoremC3}{Theorem C3}
\newtheorem*{theoremC4}{Theorem C4}
\newtheorem*{theoremD}{Theorem D}

\newtheorem*{lemmaA1}{Lemma A1}

\numberwithin{equation}{section}

\newcommand{\mc}[3]{\multicolumn{#1}{#2}{#3}}

\newlength{\mkom}
\newcommand{\h}{\hline}

\newcommand{\s}[1]{{\footnotesize #1}}

\newcommand{\mcf}{\multicolumn{1}{c}{}}

\newcommand{\Inc}[0]{\parallel}%

\newcommand{\Z}{{\ensuremath{\mathbb{Z}}}}
\newcommand{\R}{{\ensuremath{\mathbb{R}}}}
\newcommand{\C}{{\ensuremath{\mathbb{C}}}}
\newcommand{\F}{{\ensuremath{\mathbb{F}}}}
\newcommand{\G}{{\ensuremath{\mathbb{G}}}}
\newcommand{\p}{\mathscr{P}}
\newcommand{\q}{\mathscr{Q}}

\DeclareMathOperator{\Imm}{Im}%
\DeclareMathOperator{\Real}{Re}%
\DeclareMathOperator{\corep}{corep\,}%
\DeclareMathOperator{\supp}{Supp\,}%
\DeclareMathOperator{\Max}{max\,}%
\DeclareMathOperator{\Min}{min\,}%
\DeclareMathOperator{\ind}{Ind\,}%
\DeclareMathOperator{\Int}{Int\,}%
\DeclareMathOperator{\codim}{codim\,}%
\DeclareMathOperator{\dimc}{\underline{dim}\,}%
\DeclareMathOperator{\car}{Char\,}%
\newcommand{\seven}{\widehat{\text{VII}}}

\newcommand{\Mtres}{\begin{picture}(170,50) \put(1,40){$A_{25}$} 
\multiput(10.35,2)(30,0){2}{$\otimes$}
\put(42,22){$\circ$}
\put(45,8.8){\line(0,1){13.9}}
\put(20,3){{\footnotesize $a$}}
\put(50,3){{\footnotesize $b$}}
\put(50,23){{\footnotesize $\eta$}}
\put(101,40){$A_{25}^*$}
\multiput(120.35,22)(30,0){2}{$\otimes$}
\put(152,2){$\circ$}
\put(155,21.2){\line(0,-1){13.9}}
\put(130,23){{\footnotesize $a$}}
\put(160,23){{\footnotesize $b$}}
\put(160,3){{\footnotesize $\eta$}}  \end{picture}}

\newcommand{\Mocho}{\begin{picture}(170,50)  \put(1,40){$A_{38}$} 
\multiput(10.35,2)(30,0){2}{$\otimes$}
\put(25.35,22){$\otimes$}
\put(17.4,8.2){\line(3,4){10.2}}
\put(42.6,8.2){\line(-3,4){10.2}}
\put(35,23){{\footnotesize $q$}}
\put(20,3){{\footnotesize $a$}}
\put(50,3){{\footnotesize $b$}}
\put(111,40){$A_{38}^*$}
\multiput(120.35,22)(30,0){2}{$\otimes$}
\put(135.35,2){$\otimes$}
\put(127.4,21.8){\line(3,-4){10.2}}
\put(152.6,21.8){\line(-3,-4){10.2}}
\put(145,3){{\footnotesize $q$}}
\put(130,23){{\footnotesize $a$}}
\put(160,23){{\footnotesize $b$}}  \end{picture}}

\newcommand{\Mnueve}{\begin{picture}(180,50)  \put(1,40){$A_{39}$} 
\multiput(.35,2)(30,0){2}{$\otimes$}
\put(15.35,22){$\otimes$}
\put(47,22){$\circ$}
\put(7.4,8.2){\line(3,4){10.2}}
\put(32.6,8.2){\line(-3,4){10.2}}
\put(37.4,8.2){\line(3,4){11.2}}
\put(25,23){{\footnotesize $q$}}
\put(10,3){{\footnotesize $a$}}
\put(40,3){{\footnotesize $b$}}
\put(55,23){{\footnotesize $\theta$}}
\put(111,40){$A_{39}^*$}
\multiput(120.35,22)(30,0){2}{$\otimes$}
\put(135.35,2){$\otimes$}
\put(167,2){$\circ$}
\put(127.4,21.8){\line(3,-4){10.2}}
\put(152.6,21.8){\line(-3,-4){10.2}}
\put(157.4,21.8){\line(3,-4){11.2}}
\put(145,3){{\footnotesize $q$}}
\put(130,23){{\footnotesize $a$}}
\put(160,23){{\footnotesize $b$}}
\put(175,3){{\footnotesize $\theta$}}  \end{picture}}

\newcommand{\Ltres}{\begin{picture}(170,50)  \put(1,40){$A_{26}$} 
\multiput(2,2)(20,0){2}{$\circ$}
\put(40.35,2){$\otimes$}
\put(42,22){$\circ$}
\put(45,8.8){\line(0,1){13.9}}
\put(10,3){{\footnotesize $\varrho$}}
\put(30,3){{\footnotesize $\sigma$}}
\put(50,3){{\footnotesize $a$}}
\put(50,23){{\footnotesize $\eta$}}
\put(81,40){$A_{26}^*$}
\multiput(92,22)(20,0){2}{$\circ$}
\put(130.35,22){$\otimes$}
\put(132,2){$\circ$}
\put(135,21.2){\line(0,-1){13.9}}
\put(100,23){{\footnotesize $\varrho$}}
\put(120,23){{\footnotesize $\sigma$}}
\put(140,23){{\footnotesize $a$}}
\put(140,3){{\footnotesize $\eta$}}  \end{picture}}

\newcommand{\Lcuatro}{\begin{picture}(170,50)  \put(1,40){$A_{27}$} 
\multiput(2,2)(30,0){2}{$\circ$}
\put(17,22){$\circ$}
\put(60.35,2){$\otimes$}
\put(6.41,6.88){\line(3,4){12.3}}
\put(33.59,6.88){\line(-3,4){12.3}}
\put(25,23){{\footnotesize $\zeta$}}
\put(70,3){{\footnotesize $a$}}
\put(40,3){{\footnotesize $\sigma$}}
\put(10,3){{\footnotesize $\varrho$}}
\put(81,40){$A_{27}^*$}
\multiput(92,22)(30,0){2}{$\circ$}
\put(107,2){$\circ$}
\put(150.35,22){$\otimes$}
\put(96.41,23.12){\line(3,-4){12.3}}
\put(123.59,23.12){\line(-3,-4){12.3}}
\put(115,3){{\footnotesize $\zeta$}}
\put(160,23){{\footnotesize $a$}}
\put(130,23){{\footnotesize $\sigma$}}
\put(100,23){{\footnotesize $\varrho$}}
\end{picture}}

\newcommand{\Lonce}{\begin{picture}(220,50)  \put(1,40){$A_{40}$}    
\multiput(2,2)(30,0){2}{$\circ$}
\put(60.35,2){$\otimes$}
\put(45.35,22){$\otimes$}
\put(36.41,6.88){\line(3,4){11.2}}
\put(62.6,8.2){\line(-3,4){10.2}}
\put(75,3){{\footnotesize $a$}}
\put(55,23){{\footnotesize $q$}}
\put(40,3){{\footnotesize $\sigma$}}
\put(10,3){{\footnotesize $\varrho$}}
\put(131,40){$A_{40}^*$}
\multiput(142,22)(30,0){2}{$\circ$}
\put(200.35,22){$\otimes$}
\put(185.35,2){$\otimes$}
\put(176.41,23.12){\line(3,-4){11.2}}
\put(202.6,21.8){\line(-3,-4){10.2}}
\put(215,23){{\footnotesize $a$}}
\put(195,3){{\footnotesize $q$}}
\put(180,23){{\footnotesize $\sigma$}}
\put(150,23){{\footnotesize $\varrho$}}
\end{picture}}

\newcommand{\US}[1]{\makebox[1.2\width]{\mbox{$#1$}\makebox[0pt]{\hspace{-3mm}\raisebox{1.6ex}{{\scriptsize
$\thicksim$}}}}}%
\newcommand{\UI}[1]{\makebox[1.2\width]{\mbox{$#1$}\makebox[0pt]{\hspace{-4mm}\raisebox{-0.8ex}{{\scriptsize
$\thicksim$}}}}}%
\newcommand{\bsm}{\begin{smallmatrix}}
\newcommand{\esm}{\end{smallmatrix}}
\newcommand{\bM}{\begin{matrix}}
\newcommand{\eM}{\end{matrix}}
\newcommand{\bt}{\begin{tabular}}
\newcommand{\et}{\end{tabular}}
\newcommand{\bs}{\begin{subarray}}
\newcommand{\es}{\end{subarray}}

\title{One-parameter 2-equipped posets and classification of their corepresentations}
\author{Claudio Rodríguez Beltrán\thanks{2010 Mathematical Subject Classification 16G20; 16G60; 06A11. The author thanks the financial support received from project CONACYT-México 81948. Present address: Universidad Nacional de Colombia, Bogotá; \emph{e-mail address: crodriguezbe@unal.edu.co}} \\ Universidad Nacional de Colombia}

\date{\today}

\begin{document}

\maketitle
\begin{abstract}
In this paper a one-parameter criterion for 2-equipped posets with respect to corepresentations is stated and proved. The list of sincere one-parameter 2-equipped posets is given as well as a complete matrix classification of all their indecomposable corepresentations. Relations between dimensions of indecomposable corepresentations and roots of the Tits quadratic form are also established.
\end{abstract}

\section{Introduction}\label{introduction}

Representations of posets (partially ordered sets) and quivers arose in a natural way from the study of important problems of finite dimensional algebras, see \cite{ARS, Gab-Roi-92, Ri84, Roi-68, Sim92}.

In the early 1970s, Nazarova and Roiter introduced representations of posets in matrix language \cite{Naz-Roi-72}. Later Gabriel introduced representation of quivers \cite{g72},  studied further representations of some special posets and gave an invariant interpretation of representations of posets in terms of linear spaces \cite{g73}. One of the main considered problems in the representation theory of posets is to classify indecomposable representations up to isomorphism.

In order to solve these problems, matrix and combinatical techniques were developed, namely the algorithms of differentiation with respect to a maximal point \cite{Naz-Roi-72} and to a suitable pair of points \cite{Zav-77}. The first one was used to obtain the finite type criterion \cite{Kleiner-72a} and a description of their representations \cite{Kleiner-72b}, while the second was used to proof the criterion of finite growth type \cite{Zav-Naz-82}. Another important classification results are the tame and the one-parameter type criterion, obtained in \cite{Naz-75, Otr-76}, respectively.

After this exhaustive development, the researches focused on representations of posets with additional structures such as posets with involution \cite{Naz-Roi-83}, with an equivalent relation \cite{BZ-91, Naz-Roi-83, Naz-Roi-02, Zav-91}, dyadic posets \cite{Naz-Roi-00, Roi-Bel-Naz-96}, valued posets \cite{Kle-Sim-90} as well as 2-equipped posets (formerly called equipped posets) introduced in \cite{Zab-Zav-99}. In particular, for representations of 2-equipped posets it holds an interpretation as a schurian subcategory of a vector space category, see \cite{Sim92}.

In the last decade, many classification results about representation of 2-equipped posets were obtained such as the criterion of one-parameter, tame, and finite growth type, see \cite{Zab-Zav-99, Zav-03-TEP, Zav-05-EPFG}, respectively. Some of these classification problems can be reduced to a matrix problems of mixed type over the pair of fields of real and complex numbers $(\R , \C )$. Precisely, in this context arise another matrix problem for 2-equipped posets which behaves in some sense dual to the representation one, and leads to the study of their corepresentations, introduced in \cite{Rod-Zav-07}.

Then, it is important to develop a systematic theory of corepresentations and reduction algorithms of 2-equipped posets. Considering the results in \cite{Kle-Sim-90}, it must be considered the 2-equipped posets of infinite type, at first, those of one-parameter type. As the representation case has shown, the classification problem of indecomposable corepresentations of one-parameter type is basic and quite deep among tame type problems. According to the logical development of the theory of representations is necessary and indispensable to obtain a criterion for deciding when a 2-equipped poset is of one-parameter with respect to its corepresentations and then, obtain a complete classification of their indecomposable corepresentations (even the corresponding series) besides the list of sincere posets.

In order to reach these objectives, in the present paper, it is used both linear algebraic and combinatorial methods of the classic theory of representations of posets as well as new techniques, in particular, the algorithm of differentiation $\seven$ for matrix problems of mixed type over an arbitrary quadratic field extension $\F\subset \G$, developed in \cite{Rod-Zav-07}, and special reduction methods together with structural properties of corepresentations of some critical posets, obtained in \cite{Rod-10}. The interested reader may consult also the Ph.D. thesis \cite{Rod-12} in which is gathered the most important information concerning to this investigation.

The reader who does not know about corepresentations of 2-equipped posets can get familiar with them in Section \ref{definitions and notations}, which is a short introduction to the basic facts and it also recalls some definitions and notations. Finally, the list of 2-equipped posets of finite type and their classification is given in appendix \ref{finitetypeposets}.

The main goals of this research (formulated in Section \ref{main theorems}) are to obtain the criteria of one-parameter type for 2-equipped posets with respect to corepresentations (Theorem A), the list of sincere one-parameter 2-equipped posets (Theorem B) along with their respective diagrams which are presented in appendixes \ref{sincereK6posets}, \ref{sincereK8posets} and \ref{specialposets}, the classification of indecomposable corepresentations of the critical one-parameter 2-equipped posets (Theorems C1 - C4)\footnote{Theorems C1 and C3 (in the present paper) have already been obtained in \cite{Rod-10}, in order to establish some structural properties about 2-equipped posets containing the critical ones $K_6$ and $K_8$.}, and Theorem D which establishes a relationship between the Tits quadratic form associated to a 2-equipped poset for corepresentations and the dimension vectors of the indecomposable corepresentations of one-parameter 2-equipped posets.

An explicit description of the indecomposable corepresentations of each sincere one-parameter 2-equipped poset, even the critical ones, is given in Sections \ref{sectionK6} - \ref{sectionK9}, while proof of Theorems A, B and D are left in Section \ref{proofs}. Finally, the minimal dimension vectors corresponding to indecomposable corepresentations of sincere one-parameter 2-equipped posets are listed in appendixes \ref{dimensions} and \ref{dimensions of K7 and K9}.

\section{Basic Definitions and Notations}\label{definitions and notations}

\subsection{2-equipped posets}

A \textbf{2-equipped poset} is a triple $( \p ,\ \le, \ \lhd )$ where $(\p,\ \le )$ is an ordinary poset, and $\lhd$ is an additional binary relation on $\p$ contained in $\le$. The \textbf{strong} relation $\lhd$ also satisfies the following condition:
\begin{equation}\label{New def. of equipped poset}
    x\leq y\lhd z\quad or\quad x\lhd y\leq z\quad implies\quad x\lhd z.
\end{equation}

Whether $x\le y$ and $x \ntriangleleft y$, we write $x\prec y$ and relation $\prec$ is called \textbf{weak}. A point $x\in \p$ is called \textbf{strong} (\textbf{weak}) if $x\lhd x$ ($x\prec x$). Notice that $\le$ is a disjoint union of $\lhd$ and $\prec$, and both relations need not be a partially ordered relations on $\p$.

A 2-equipped poset $( \q ,\ \le_{\q}, \ \lhd_{\q} )$ is a \textbf{subposet} of $( \p ,\ \le_{\p}, \ \lhd_{\p} )$ if $\q\subset\p$ is a full subset, it means $\; x\le_{\q} y \ \Longleftrightarrow \ x\le_{\p} y\; $ and $\; x\lhd_{\q} y \ \Longleftrightarrow \ x\lhd_{\p} y\; $ for all $x,y\in\q$. From now on, all considered posets are finite, and for the brevity, we will write simply $\p$ instead of $(\p,\ \le, \ \lhd )$. We write $x \Inc y$ if the points are incomparable. For any subset $X\subset \p $ denote
\begin{equation}\label{incomparable}
N(X) = \{ a\in \p : a\Inc x \text{ for all }x \in X \}.
\end{equation}

Given two subposets $X, Y\subset \p$, we denote by $X + Y$ their \textbf{cardinal sum} or \textbf{sum} (possibly with elements comparable between them), if it is further required that $x \le y$ for all $x\in X$ and $y\in Y$ it means their \textbf{ordinal sum} and is denoted by  $X\le Y$; The \textbf{strong} and \textbf{weak ordinal sums} $X\lhd Y$ and $X\prec Y$ are analogous and respectively defined. Let $|X|$ be the cardinality of a set $X$. Unless there confusion, we simply write a singleton set $\{ x \}$ as $x$.

For a subposet $X\subset \p$, set $\; X^{\vee}= \{ y\in\p \ : \ x\le y \; \text{for some} \; x\in X \}$; $X^{\triangledown}, \ X^{\curlyvee}$ are analogously defined, and $X_{\wedge}, \ X_{\triangle}, \ X_{\curlywedge}$ are dually defined, finally $X^{\blacktriangledown} = X^{\vee}\setminus X$, $X_{\blacktriangle} = X_{\wedge}\setminus X$. We say that some relation $xRy$ between two points is \textbf{strict} if $x\not= y$. The set of all maximal (minimal) points in $X\subset \p$ is denoted by $\Max X$ ($\Min X$).

A subposet of $\p$ is a \textbf{chain} (\textbf{anti-chain}) if all its points are pairwise comparable (incomparable), in particular, a \textbf{dyad} (triad) is an anti-chain of two (three) points, a \textbf{garland} is an ordinal sum of singleton sets and dyads. The \textbf{length} of a chain is the number of its points. A chain of the form $a_1 \prec a_2 \prec \cdots \prec a_n$ is called \textbf{weak}, if additionally $a_1 \prec a_n$ then it is \textbf{completely weak}. An arbitrary subset $X\subset\p$ is said to be \textbf{ordinary} (\textbf{completely weak}) if all its points and possible relations between them are strong (weak), in particular, if $\p$ just has strong points is called ordinary and its equipment is trivial.

The \textbf{weight} $w(x)$ of a point $x\in\p$ is $1$ if $x$ is a strong point and $2$ if it is a weak point. Whether $X\subset\p$ is an anti-chain, its \textbf{weight} is defined as the sum of the weights of each of its elements and it is denoted by $w(X) = \sum_{x\in X} w(x)$. The \textbf{weight} of $\p$ is $w(\p) = \max\{ w(X)\: : \: \text{for all anti-chains } X \subset \p \}$.

\begin{figure}[h]
\centering
\begin{picture}(140,100)
\put(2,2){$\circ$}
\put(6.5,6.5){\line(1,1){25.7}}
\multiput(30.35,32)(30,30){2}{$\otimes$}
\multiput(60.35,2)(30,30){2}{$\otimes$}
\put(90.35,92){$\otimes$}
\put(122,62){$\circ$}
\put(97.8,37.8){\line(1,1){25.6}}
\multiput(37.8,37.8)(30,30){2}{\line(1,1){24.4}}
\qbezier(35,39)(50,80)(91,95)
\put(67.3,8.3){\line(1,1){24.4}}
\put(68.3,7.3){\line(1,1){24.4}}
\put(37.8,32.2){\line(1,-1){24.4}}
\put(12,2){{\footnotesize $\alpha$}}
\put(132,62){{\footnotesize $\beta$}}
\put(72,2){{\footnotesize $a$}}
\put(42,32){{\footnotesize $b$}}
\put(72,62){{\footnotesize $c$}}
\put(102,92){{\footnotesize $d$}}
\put(102,32){{\footnotesize $e$}}
\end{picture}
\caption{Diagram of a 2-equipped poset.}\label{equipped poset}
\end{figure}
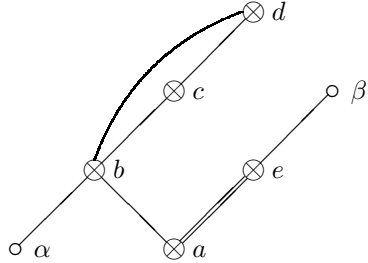

Graphically each 2-equipped poset is presented by its \textbf{diagram}, which is obtained from the ordinary Hasse diagram of the poset $\p$ by distinguishing its strong and weak points by $\circ$ and $\otimes$ respectively, and by joining additional lines symbolizing those strong strict relations between weak points which are not consequences of other relations. It follows from (\ref{New def. of equipped poset}) that relations which rise and sink from strong points are strong, so these are not distinguished in diagrams.

\begin{example}
Let $\p$ be an equipped poset given by the diagram of Figure \ref{equipped poset}. Then, among strict relations, the only weak ones are $a\prec b \prec c \prec d, \:\:\: a\prec c $, hence all those in the rest are strong, namely,  $ b \lhd d, \:\:\: \alpha \lhd \{b,c,d\}, \:\:\: a \lhd \{ d, e, \beta \}, \:\:\: e\lhd \beta$. On the other hand, $\p = \alpha^{\vee} + \beta_{\wedge}$ is the sum of two chains, and $w(\p) = 4$.
\end{example}

\subsection{The category $\corep \p$ of corepresentations}

Unless otherwise stated, throughout this paper, $\F\subset \F (\xi) = \G$ denotes an arbitrary quadratic field extension with primitive element $\xi$. Since each $x\in \G$ is uniquely expressed as $\F$-linear combination $x = a + \xi b$, its \textbf{real} and \textbf{imaginary part} $\Real (x) = a$ and $\Imm (x) = b$ are well defined. Given a 2-equipped poset $\p$, a \textbf{matrix corepresentation} $M$ is a rectangular matrix over $\G$ separated into matrix blocks $M_x$ ($x\in\p$) horizontally arranged, $M = \left[ M_{x_1} | \cdots | M_{x_n} \right]_{(x_i\in\p)}$, along with a set of \textbf{admissible transformations}:
\begin{enumerate}
\item [$(a)$]$\G$-elementary row transformations of the whole matrix $M$;
\item [$(b)$]$\G$-elementary ($\F$-elementary) column transformations of a stripe $M_x$ if the point $x$ is strong (weak);
\item [$(c)$]In the case of a strong (weak) relation $x\lhd y$ ($x\prec y$), additions of columns of the stripe $M_x$ to the columns of the stripe $M_y$ with coefficients in $\G$ ($\F$).
\end{enumerate}

Let $M, \ N$ be a matrix corepresentations of $\p$, its \textbf{direct sum} is defined in a natural way as $M\oplus N = [M_x\oplus N_x]_{(x\in\p)}$. $M$ is said to be \textbf{equivalent} or \textbf{isomorphic} to $N$ if they can be turned into each other by means of admissible transformations, and it is denoted by $M\sim N$. An \textbf{indecomposable} matrix corepresentations is defined in the usual matrix sense. The set of all corepresentations of $\p$ over ($\F, \G$) together with admissible transformations lead to the \textbf{matrix coproblem} of classify all indecomposable corepresentations up to equivalence.

To every matrix corepresentation $M$ is associated its \textbf{dimension} $d = \dimc M = (d_0, d_x\: : \: x\in\p )$ where $d_0$ ($d_x$) is the number of rows in $M$ (of columns in $M_x$).

There exists another definition of corepresentations of 2-equipped posets, in invariant language, compatible with the matrix definition above. We recall this and another facts from \cite{Rod-10, Rod-Zav-07}. Given a $\F$-subspace $U$ of some linear $\G$-space $V$, its $\G$-\textbf{hull} $\US{U}$ is the least $\G$-subspace containing $U$ and its $\G$-\textbf{cohull} $\UI{U}$ is the largest $\G$-space contained in $U$, then they can be defined as follows \begin{equation} \US{U} = \sum_{x\in U} \G x = \bigcap_{U\subset W}W \hspace{0.5cm} \text{and} \hspace{0.5cm} \UI{U} = \{ x\in U\: : \: \G x\subset U\} = \sum_{W\subset U} W \end{equation}where $W$ are $\G$-subspaces of $V$. When $U$ itself is a $\G$-subspace then $\US{U}=\UI{U}=U$ and $U$ is said to be $\G$-\textbf{strong} or \textbf{strong}.

A \textbf{corepresentation} of a 2-equipped poset $\p$ over the pair $(\F ,\G )$ is a collection of the form $U=(U_0,U_x\::\:x\in\p\, )$ where $U_0$ is a finite-dimensional $\G $-space containing $\F $-subspaces $U_x$ such that
\begin{equation}\label{Cond. for Corep.}
\begin{aligned}
x\leq y & \Longrightarrow U_x\subset U_y\,,  \\
x\lhd y & \Longrightarrow  \US{U}_x\subset U_y\,.
\end{aligned}
\end{equation}
Then, to each strong point corresponds a $\G$-strong subspace.

The \textbf{dimension} of a corepresentation $U$ of $\p$ is a vector $d=\dimc U=(d_0, d_x\::\:x\in\p)$ with $d_0=\dim_\G U_0 \;$ and $\; d_x=\dim_\G U_x/\underline{U_x}\; $ if $x$ is strong or $\; d_x=\dim_\F U_x/\underline{U_x}\; $ if $x$ is weak, where $\; \underline{U_x} = \sum_{\bs{l} y\prec x \es} U_y +\sum_{\bs{l} y\lhd x  \es}\US{U}_y \;$ with $\; y \not= x\; $, is the \textbf{radical subspace} of $U_x$.

The \textbf{category of corepresentations} $\corep_{(\F ,\G )} (\p , \le , \lhd )$, or simply $\corep \p$ if the pair ($\F, \G$) causes no confusion, has as objects the corepresentations of $\p$ and morphisms $ U \overset{\varphi}{\longrightarrow} V$ the $\G $-linear maps $\varphi : U_0\longrightarrow V_0 $ such that $\varphi(U_x)\subset V_x$ for each $x\in \p$, if further required that $\varphi$ is an isomorphism and $\varphi(U_x) = V_x$, then $U, V$ are said to be $\textbf{isomorphic}$.

A \textbf{direct sum} of two corepresentations $U, V$ of $\p$ is a corepresentation $U\oplus V = (U_0 \oplus V_0, U_x \oplus V_x \: : \: x\in\p)$. $U$ is a \textbf{indecomposable corepresentation} if its direct summands are only the null corepresentation $0 = (0, 0_x \: : \: x\in\p )$ and $U$. A complete set of indecomposable corepresentations pairwise non-isomorphic of $\p$ over the pair ($\F , \G$) is denote by $\ind_{(\F , \G)}\p$ or $\ind \p$, for brevity. A 2-equipped poset is said to be of \textbf{finite (infinite) corepresentation type} if $|\ind\p|$ is finite (infinite). The list of finite corepresentation type posets and their classification of indecomposables can be deduced from \cite{Kle-Sim-90}. The classification of indecomposable objects of the category $\corep \p$, up to isomorphism, corresponds precisely to the described above matrix problem.

Namely, if $M$ is a matrix corepresentation of $\p$ with $\dimc M = (d_0, d_x \::\:x\in\p)$, one may consider a base $e_1,\dots,e_{d_0}$ of some $d_0$-dimensional $\G$-space $U_0$ and identify each column $(\lambda_1,\dots,\lambda_{d_0})^T$ of $M$ with the element $u=\lambda_1e_1+\dots+\lambda_{d_0}e_{d_0}\in U_0$. Given any column set $X\subset M$, denote by $\F [X]$ and $\G [X]$ the $\F$-span and $\G$-span of $X$ in $U_0$ respectively. Then, one can form a corepresentation $U_M = (U_0, U_x\::\: x\in\p)$ where $U_x=\sum_{y\prec x}\F [M_y]+\sum_{y\lhd x}\G [M_y]$, which satisfies the conditions (\ref{Cond. for Corep.}).

From this point of view, the columns of each vertical stripe $M_x$ represent a system of generators of the space $U_x$ modulo its radical subspace \[ \underline{U_x} = \sum_{\bs{l} y\prec x \\ y\not= x \es} \F [M_y] +\sum_{\bs{l} y\lhd x \\ y\not= x \es}\G [M_y].\]

Hence, the transformations $(a)$-$(c)$ of $M$ reflect both base changing in $U_0$ and generator changing in subspaces $U_x$. Conversely, starting from $U$, one can associate with $U$ a matrix corepresentation denoted by $M_U$. Notice that $\dimc U_M \leq\dimc M$ and $\dimc U \leq\dimc M_U$, in both cases the equality holds if and only if the columns of each stripe $M_x$ are linearly independent modulo the radical column. When dealing with matrix corepresentations $M_U$ corresponding to corepresentations $U$, we always will select such matrices $M=M_U$ that $\dimc M = \dimc U$.

Given any corepresentation $U$ of $\p$ with dimension $d = \dimc U$, the \textbf{support} of $U$ is the subset $\supp U = \{ x\in \p : d_x > 0 \}$. A vector $d$ is \textbf{sincere} if has no zero coordinates; a corepresentation is said to be \textbf{sincere} if its dimension vector is sincere. Every equipped poset having at least one sincere indecomposable corepresentation is called \textbf{sincere} (with respect to corepresentations). A corepresentation is \textbf{trivial} if $\dim_\G U_0 = 1$.

Let $\q \subset \p$ be a subposet, and $U\in \corep \q$, the \textbf{extended corepresentation} $V\in\corep \p$ of $U$ is defined as follows $V_0 = U_0$ and for $x\in\p$

\[ V_x =
\begin{cases}
\sum\limits_{y\prec x} U_y + \sum\limits_{y\lhd x} \US{U}_y \; \text{ for } y\in\q, \; \text{ if } x\in \q^{\vee}; \\
0, \; \text{ otherwise}.
\end{cases}\] Notice that $V_x = U_x$ for $x\in\q$. Moreover if $U\overset{\varphi}{\longrightarrow} U'$ is a morphism in $\corep\q$ then the same $\G$-linear application $\varphi$ induce a morphism $V\overset{\varphi}{\longrightarrow} V'$ in $\corep \p$ between the extended corepresentations.


We have a \textbf{duality principle} for corepresentations of 2-equipped posets. Given a finite dimensional $\G$-space $U_0$, the $\F$-dual space \[ \; U_0^* = Hom_{\F}(U_0, \F)\; \]becomes a $\G$-space in a natural way. Namely, let $\wp (t) = t^2 + pt + q$ be the minimal polynomial of the quadratic field extension $\F\subset \F (\xi ) = \G$, and let  $e_1, \ldots , e_n$ a basis of $U_0 = \G\{ e_1, \ldots , e_n \}$. Then, it can be also considered as a $\F$-space of dimension $2n$ $ U_0 = \F\{ e_1, \xi e_1, \dots , e_n, \xi e_n \} $ then, if $\varphi \in U_0^*$ and it holds $\varphi (e_i) = a_i$ and $\varphi (\xi e_i) = q b_i$, we set $c_i = a_i + \xi b_i$ and $\varphi = \varphi_{(c_1, \ldots , c_n)}$ with the element $(c_1, \ldots , c_n)\in \G^n$ itself and hence identify $U_0^*$ with $\G^n$. So, $U_0^*$ is a $\G$-space. If $u = \sum_{i=1}^{n}(u_ie_i + u_i'\xi e_i)\in U_0$ (with $u_i, u_i' \in\F$), then \[ \varphi_{(c_1, \ldots , c_n)}(u) = \sum\limits_{i=1}^{n} (u_ia_i + u_i'qb_i). \]

Further, for a number $z = x + \xi y \in \G $ ($x, y \in\F$) we have $z\varphi_{(c_1, \ldots , c_n)} = \varphi_{(zc_1, \ldots , zc_n)}$. Denote $\hat{z} = x - \xi y$, notice that $\hat{z} \not= \overline{z} = x - (p + \xi)y \; $ if $\; p\not = 0$. Therefore, it holds $(z\varphi_{(c_1, \ldots , c_n)})(u) = \varphi_{(c_1, \ldots , c_n)}(\hat{z}u) $ when $p = 0$ or char $\F = 2$.

If $U_x$ is a $\F$-subspace of $U_0$, then \[ \; U_x^{\perp} = \{ \varphi \in U_0^* \, : \, U_x \subset Ker\varphi \}\; \] is a $\F$-subspace of $U_0^*$. Besides, $U_x^{\perp}$ is a $\G$-subspace provided that $U_x$ is itself a $\G$-subspace, it also holds the coincidence $U_x^{\perp\perp} = U_x$ under the natural identification $U_0^{**} = U_0$. It is summarized in the following proposition.

\begin{prop}\label{duality principle}
If, $p = 0$ or char $\F = 2$ then for any $\G$-space $ U_x\subset U_0 $ its orthogonal complement \[ U_x^{\perp} = \{ \varphi\in U_0^*\: : \: \varphi (U_x) = 0\} \] is also an $\G$-space of the dual $\G$-space $U_0^*$. Moreover, $U_x^{\perp\perp} = U_x$ under the natural identification $U_0^{**} = U_0$.
\end{prop}

Denote by $\p^* $ the antiisomorphic or dual 2-equipped poset to a poset $\p$. Given a corepresentation $U$ of $\p$, we define the \textbf{dual} corepresentation $U^*$ of $\p^*$ by setting
\[ U^* = (U_0^*, \, U_x^{\perp} \, :  \, x\in\p^* ),\]
where $U_0^*$ and $U_x^{\perp}$ are defined as above.

We also get the following properties for $\F$-subspaces $U_x, U_y$ of a $\G$-space $U_0$. $U_y^{\perp} \subset U_x^{\perp}$ if and only if $U_x\subset U_y$; and $(\UI{U}_x)^{\perp} = (\US{U}_x^{\perp})$, and $(\US{U}_x)^{\perp} = (\UI{U}_x^{\perp}) $.

\begin{lemma}\label{dualsincere}
An arbitrary 2-equipped poset $\p$ is sincere if and only if $\p^*$ is.
\end{lemma}

\begin{proof}
If $U$ is an indecomposable corepresentation of the poset $\p$ with dimension $d=\dimc U$, and $d^* = \dimc U^*$ then by changing each of the spaces $U_x$ with the condition $d_x^* = 0$ for its radical space $\underline{U_x}$, we obtain a corepresentation $U_1$ of $\p$ with codimension vector $d_1^* = \dimc U_1$ such that $(d_1^*)_x \not= 0$ for all $x\in\p$, this mean that $U_1^*$ is a sincere indecomposable corepresentation of $\p^*$
\end{proof}


It may be analogously defined a matrix corepresentation of a 2-equipped poset $\p$ over the pair of polynomials rings $(\F [t], \G [t])$. Namely, a matrix $(\F [t] , \G [t])$-corepresentation of $\p$ is a rectangular matrix $M[t]$ over $\G [t]$ separated in vertical stripes $(M_x[t])$ ($x\in\p$). However, for the scope of this paper, it is not necessary to define the analogue of the admissible transformations (a)-(c). A \textbf{corepresentation series} $S(X)$ is obtained from a matrix $M[t]$ of $(\F [t], \G [t])$-corepresentation by substituting any square $\G$-matrix $X$ for the variable $t$, and scalar matrices $\lambda I$ of the same size for the coefficients $\lambda \in \G $.

A corepresentation series of $\p$ is called \textbf{sincere} if it induces at least a sincere indecomposable corepresentation of $\p$. We are interested in series of indecomposable corepresentations pairwise non-isomorphic, then, sometimes is necessary to impose additional restrictions on the matrix $X$, see Table \ref{example of series}.

\begin{table}[hbtp]
\centering
\renewcommand{\arraystretch}{1}%
\renewcommand{\tabcolsep}{1mm}%
\setlength{\doublerulesep}{0.14pt} \small%
$S(x) = $ \begin{tabular}[c]{*{4}{|c}|} \multicolumn{1}{c}{\scriptsize{a}} &%
\multicolumn{1}{c}{\scriptsize{p}} &\multicolumn{1}{c}{\scriptsize{q}}&%
\multicolumn{1}{c}{\scriptsize{$\theta$}} \\ \h%
 \s{$\:I\:$}& \s{$X$} & \s{$\:I\:$} & \s{$\:I\:$}\\ \h%
 \s{$I$}&\s{$\xi I$}&\s{} & \s{} \\ \h%
 \multicolumn{4}{c}{(a)} \\
\end{tabular} \hspace{1cm} \small%
\begin{tabular}{*{4}{|c}|} \multicolumn{1}{c}{\scriptsize{a}} &%
\multicolumn{1}{c}{\scriptsize{p}} &\multicolumn{1}{c}{\scriptsize{q}}&%
\multicolumn{1}{c}{\scriptsize{$\theta$}} \\ \h%
 \s{$\:1\:$}& \s{$t$} & \s{$\:1\:$} & \s{$\:1\:$}\\ \h%
 \s{$1$}&\s{$\xi$}&\s{0} & \s{0} \\ \h%
 \multicolumn{4}{c}{(b)} \\
\end{tabular}
\caption{The series of ($\F , \G$)-corepresentations (a) of $K_7 = \{ a\prec p \prec q; \; a\prec q; \; \theta \}$ comes from the ($\F [t], \G [t]$)-corepresentation (b) of $\p$. The $I$ blocks correspond to identity matrix of arbitrary fixed size, and the empty blocks are zero matrices. Whether the block $X$ is a square matrix in a Frobenius Canonical form over $\F$, under ordinary similarity transformations, it holds that $ X \sim X' \Longleftrightarrow S(X) \sim S(X') $. We will resume its discussion later in Section \ref{sectionK7}.}\label{example of series}
\end{table}%

\begin{definition}
Let $\F$ be a infinite field. A 2-equipped poset $\p$ of infinite corepresentation type is of \textbf{one-parameter} with respect to corepresentations over the pair $(\F, \G)$, if it has a series containing almost all indecomposable corepresentations of each given dimension
\end{definition}

\section{Formulation of main theorems}\label{main theorems}

\noindent

From now on all results are formulated without warning for corepresentations (if there is not confusion).

\begin{theoremA}\label{criterion}
Let $\p$ be a 2-equipped poset with $w(\p)\le 4$. Then, $\p$ is one-parameter type if and only if it has precisely a critical subposet, i.e. one of the posets $K_1,\ldots , K_5$ (the  Kleiner's critical posets) and $K_6, \ldots , K_9$ (the critical non trivially equipped posets) listed in Figure \ref{critical equipped posets}.
\end{theoremA}

\begin{figure}[h]
\centering
\begin{picture}(400,80)
\put(36,1){$K_6$}
\multiput(20.35,22)(30,0){2}{$\otimes$}
\put(30,23){{\footnotesize $a$}}
\put(60,23){{\footnotesize $b$}}

\put(126,1){$K_7$}
\multiput(110.35,22)(0,20){3}{$\otimes$}
\put(142,22){$\circ$}
\multiput(115,28.8)(0,20){2}{\line(0,1){12.5}}
\put(120,63){{\footnotesize $q$}}
\put(120,43){{\footnotesize $p$}}
\put(120,23){{\footnotesize $a$}}
\put(150,23){{\footnotesize $\theta$}}

\put(231,1){$K_8$}
\multiput(202,22)(30,0){2}{$\circ$}
\put(260.35,22){$\otimes$}
\put(270,23){{\footnotesize $a$}}
\put(240,23){{\footnotesize $\sigma$}}
\put(210,23){{\footnotesize $\varrho$}}

\put(336,1){$K_9$}
\multiput(320.35,22)(0,20){2}{$\otimes$}
\multiput(352,22)(0,20){2}{$\circ$}
\put(325,28.8){\line(0,1){12.5}}
\put(355,26.9){\line(0,1){16.2}}
\put(330,43){{\footnotesize $p$}}
\put(330,23){{\footnotesize $a$}}
\put(360,43){{\footnotesize $\sigma$}}
\put(360,23){{\footnotesize $\varrho$}}
\end{picture}
\caption{Diagrams of the critical one-parameter 2-equipped posets.}\label{critical equipped posets}
\end{figure}
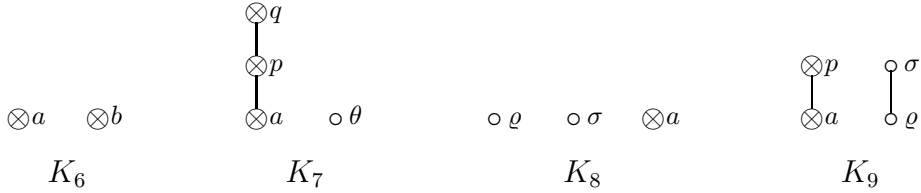

It follows that if $w(\p) \le 3$ of $\p$ does not contain critical 2-equipped poset then is of finite corepresentation type. The list of non trivially 2-equipped posets is given in the Appendix \ref{finitetypeposets}, besides of their classification of indecomposable corepresentations.

\begin{theoremB}\label{sincereposets}
A non trivially 2-equipped poset of one-parameter type is sincere if and only if it is isomorphic or antiisomorphic to one of the 28 posets $K_6,\ldots , K_9$, and $A_{25}, \ldots , A_{48}$ listed in appendixes \ref{sincereK6posets}, \ref{sincereK8posets} and \ref{specialposets}.
\end{theoremB}

An explicit classifications of indecomposable corepresentations of the critical posets $K_6$ and $K_8$ were obtained in \cite{Rod-10}, however we recall them in Theorems C1 and C3. The classifications of indecomposables of $K_7$ and $K_9$ are given in Theorems C2 and C4 respectively. The remaining classifications of indecomposable corepresentations of the sincere 2-equipped posets are given in propositions \ref{Classification of coreps of A25}-\ref{A42coreps}, and \ref{A26A27coreps}, \ref{A40coreps}, \ref{A30'scoreps} and \ref{A43'scoreps}.

\begin{theoremC1}\label{Theorem C1}
The critical 2-equipped poset $K_6 = \{a; \: b \}$ where $a; \: b$ are weak points, has 6 types of indecomposable corepresentations over the pair ($\F ,\ \G$) (up to duality and permutation), even the series of corepresentations, see Tables \ref{series of coreps of K6} and \ref{List of coreps of A25}.
\end{theoremC1}

\begin{theoremC2}\label{Theorem C2}
The critical 2-equipped poset $K_7 = \{ a\prec p \prec q; \: a\prec q; \: \theta \}$ where $\theta$ is the unique strong point, has 48 types of indecomposable corepresentations over the pair ($\F ,\ \G$) with Tits form not 0, it also has a corepresentation series.
\end{theoremC2}

\begin{theoremC3}\label{Theorem C3}
The critical 2-equipped poset $K_8 = \{ \varrho; \: \sigma; \: a \}$ where $a$ is the unique weak point, has 11 types of indecomposable corepresentations over the pair  ($\F, \G$) (up to duality and permutation), even the series of corepresentations, .
\end{theoremC3}

\begin{theoremC4}\label{Theorem C4}
The critical 2-equipped poset $K_9 = \{ a\prec b; \: \zeta \lhd \eta \}$ where $a, b$ are weak points and $\zeta$, $\eta$ are strong points, has 48 types of indecomposable corepresentations over the pair ($\F, \G$) with Tits form not 0, it aslo has a corepresentation series.
\end{theoremC4}

\subsection{The Tits quadratic form asociated to a 2-equipped poset}\label{TitsForm}
\noindent
There exists a known relation between the theory of quadratic forms and the representation theory of posets, see \cite{Ri84}, and as expected there also exists a relation with corepresentations of 2-equipped posets. Namely, it is defined a quadratic Tits form $f_{\p}(d)$ of a vector $d$ associated to a 2-equipped poset $\p$. Furthermore, there is relationship between roots of $f_{\p}$ and dimension vectors of indecomposable corepresentations of $\p$ (see Theorem \ref{roots}).

Given a 2-equipped poset $\p$, we set $\p^{\bullet} = \p\cup \{0\}$ (where 0 is a formal symbol incomparable with the elements of $\p$). A vector $e_x\in\Z^{\p^{\bullet}}$ ($x\in \p^{\bullet}$) is a \textbf{simple root} whether $(e_x)_x = 1$ and $(e_x)_y = 0$ for $x\not=y$. The \textbf{Tits quadratic form} $f_{\p} = f$ associated to $\p$ is defined over a given vector $d = (d_0,\ d_x\ : \ x\in\p )$ by the formula \[  f(d) = \sum\limits_{ x,y\in \p^{\bullet} }  l_{xy}d_xd_y    - 2d_0\sum\limits_{x\in\p} d_x, \]
where $l_{xy} = 0 $, if  $x\Inc y$; $l_{xy} = 1 $, if  $x\prec y$; $l_{xy} = 2 $, if  $x\lhd y$ or $x=y=0$. Recall that if $\langle x,y \rangle = \frac{1}{4}[f(x+y)-f(x-y)]$ is the corresponding symmetric bilineal form, the \textbf{reflection} in the point $x\in\p^{\bullet}$ is an application $\rho_x: \Z^{\p^{\bullet}} \longmapsto \Z^{\p^{\bullet}}$ defined by $\rho_x(d) = d - \frac{2}{l_{xx}}\langle d,e_x \rangle \cdot e_x. $

A \textbf{root} of a Tits form $f_{\p}$ is a vector $d$ associated to $\p$ which is obtained by reflections from a simple root. Let $d$, $d'$ be roots of $f_{\p}$, we write $d\le d'$ if $d_x \le d'_x$ ($x\in\p^{\bullet}$). Throughout, we consider \textbf{admissible roots} which are non-negative roots, with $d_0 > 0$. A vector $\mu \ge \bf{0}$ associated to $\p$ is an \textbf{imaginary root} of $f_{\p}$ if $f_{\p}(\mu) = 0$. A vector which is neither root nor imaginary root is called \textbf{special vector}.

The following result describes a close relation between Tits form $f_{\p}$ associated to $\p$ and its dimension vectors of indecomposable corepresentations.

\begin{theoremD}[Dimensions of indescomponibles]\label{roots}
A vector $d > 0$, corresponding to an one-parameter 2-equipped poset, is a dimension of an indecomposable corepresentation if and only if it is an admissible root, an imaginary root, or a special vector, moreover
\begin{enumerate}
\item[(a)] If $d$ is an admissible root or a special vector, then there exists precisely one (up to isomorphism) indecomposable corepresentation of dimension $d$.
\item[(b)] If $d$ is an imaginary root, then there exist (in case of infinite fields $\F$, $\G$) infinitely many non isomorphic indecomposable corepresentations of dimension $d$.
\end{enumerate}
\end{theoremD}

\begin{remark}
Sincere indecomposable corepresentations which are not covered by a series are called \textbf{discrete}. Due to Theorem D, those discrete indecomposables, whose dimension vectors are either an admissible root or a special vector, are characterized by them.
\end{remark}

Given two vectors $d = (d_0, d_x \: : \: x\in \p)$ and $d' = (d_0', d_x' \: : \: x\in \p)$ associated to $\p$, they are of \textbf{the same type} if $d - d' = \mu$ where $\mu$ is an imaginary root of $f_{\p}$. A \textbf{minimal root} $\mu_0$ of $f_{\p}$ is a minimal element of a set of admissible roots of the same type of $f_{\p}$. Given $U,\ V \in \ind\p$ with dimension vectors $d,\ d'$ respectively, which are not imaginary roots, they are said to be of \textbf{the same type} if $d$ and $d'$ are of the same type. Moreover, we attach a number $m$ to all indecomposables of the same type, and refer to an specific one by its dimension in the following way $U =\ _{d_0}(\p - m)$ and $V =\ _{d_0'}(\p - m)$.

If the dimension vectors $d, \ d'$ are imaginary roots but $U$ and $V$ are not covered by a series of corepresentations, and besides they are obtained from integration of another indecomposables of the same type, we say that $U$ and $V$ are of the same type. This is just a convention to refer in a simply way to the indecomposable corepresentations.

The dimension vectors, of the one-parameter 2-equipped posets having one type of discrete indecomposables, are presented together with their diagrams in the Appendixes \ref{sincereK6posets} - \ref{specialposets}, while in the Appendixes \ref{dimensions} and \ref{dimensions of K7 and K9} are presented the list of minimal dimension vectors both of the posets with more than one type of discrete indecomposables, and of the critical posets $K_7$ and $K_9$, respectively.

\section{Classification for one-parameter 2-equipped posets containing $K_6 = \{ \otimes \ \otimes \}$}\label{sectionK6}

We present a complete classification\footnote{Although the classification of indecomposables for $K_6$, $A_{25}$, $A_{28}$, $A_{38}$, $A_{41}$ and their dual posets have already been realized in \cite{Rod-10}, these are presented here in a convenient way for the purpose of this paper}, in matrix form, of indecomposable corepresentations of the posets $K_6$, $A_{25}$, $A_{28}$, $A_{29}$, $A_{33}$, $A_{34}$, $A_{38}$, $A_{39}$, $A_{41}$, $A_{42}$, $A_{45}$, $A_{46}$ and their dual sets over the pair $(\F , \G)$. These posets are those sincere one-parameter 2-equipped posets that contain the poset $K_6$. The classifications of indecomposables for above mentioned posets are obtained using differentiation and integration $\seven$ developed in \cite{Rod-Zav-07}, except for the series of corepresentations of $K_6$.

Let $\q,\ \p$ be a 2-equipped posets such that $\q$ is a subposet of the derived poset $\p'$ with respect to the algorithm D-$\seven$. We write \[ _{d_0}(\p - m) =  \Int_{d_0'}(\q - n), \]to mean that the corepresentation $_{d_0}(\p - m)$ comes from the (extended) corepresentation $_{d_0'}(\q - n)$ of $\p'$ by integration $\seven$. Notice that there is not a unique way to compute it.

Besides the above notation, we also use a direct way to build a corepresentation. Given a maximal (minimal) strong point $\zeta$ of a 2-equipped poset $\p$, we set $\mathcal{R} \subset \mathcal{Q} = \p\setminus \zeta$ the subset of comparable points in $\mathcal{Q}$ with $\zeta$. Let $U$ be an indecomposable corepresentation of $\mathcal{Q}$, such that $\US{U}_{\mathcal{R}}\not= U_0$ ($\UI{U}_{\mathcal{R}}\not= 0$), then $U$ may be be extended to a corepresentation $V$ of $\p$ given by $V_0 = U_0$, $V_x = U_x$ for $x\in \mathcal{Q}$ and $V_{\zeta}$ being an arbitrary $\G$-space containing properly (contained into) the space $\US{U}_{\mathcal{R}}$ ($\UI{U}_{\mathcal{R}}$) provided that $dim_{\G} U_{\zeta}/\US{U}_{\mathcal{R}} = 1$ ($dim_{\G} U_{\zeta} =  1$) and such that $V$ does not depend on $V_{\zeta}$ up to isomorphism. $V$ is denoted by $U^{\zeta}$ ($U_{\zeta}$).

\begin{example}
In table \ref{Example of corepresentation extended} is considered $\p = A_{25} = \{ a; \: b\lhd \eta \}$ and $\p^* = \{ a; \: \eta \lhd b \}$ where $\eta$ is the only strong point, and it is maximal and minimal respectively. $K_6 = \{ a; \: b \}$ is a subposet of both of them and its matrix indecomposable corepresentation $_4(K_6 - 4)$ satisfies conditions above, then it can be extended to a corepresentation of $A_{25}$ or $A_{25}^*$, moreover these extended corepresentations are indecomposables.
\end{example}
\begin{table}[h]\renewcommand{\arraystretch}{1}\renewcommand{\tabcolsep}{1mm}
\centering
\begin{tabular}[b]{|cccc|cccc|}
\multicolumn{4}{c}{\tiny{$a$}} & \multicolumn{4}{c}{\tiny{$b$}} \\ \hline
    $1$ &   $\xi$ & \s{$0$} & \s{$0$} & \s{$0$} & \s{$0$} & \s{$0$} & \s{$0$} \\
\s{$0$} &     $1$ & \s{$0$} & \s{$0$} &     $1$ &   $\xi$ & \s{$0$} & \s{$0$} \\
\s{$0$} & \s{$0$} &     $1$ &   $\xi$ & \s{$0$} &     $1$ & \s{$0$} & \s{$0$} \\
\s{$0$} & \s{$0$} & \s{$0$} &     $1$ & \s{$0$} & \s{$0$} &     $1$ &   $\xi$ \\ \hline
    \multicolumn{8}{c}{\s{$_4(K_6 - 4)$}} \\
\end{tabular} \hspace{0.7cm} \begin{tabular}[b]{|cccc|cccc|c|}
\multicolumn{4}{c}{\tiny{$a$}} & \multicolumn{4}{c}{\tiny{$b$}} & \multicolumn{1}{c}{\tiny{$\eta$}} \\ \hline
    $1$ &     $\xi$ & \s{$0$} & \s{$0$} & \s{$0$} & \s{$0$} & \s{$0$} & \s{$0$} &     $1$ \\
\s{$0$} &       $1$ & \s{$0$} & \s{$0$} &     $1$ &   $\xi$ & \s{$0$} & \s{$0$} & \s{$0$} \\
\s{$0$} &   \s{$0$} &     $1$ &   $\xi$ & \s{$0$} &     $1$ & \s{$0$} & \s{$0$} & \s{$0$} \\
\s{$0$} &   \s{$0$} & \s{$0$} &     $1$ & \s{$0$} & \s{$0$} &     $1$ &   $\xi$ & \s{$0$} \\ \hline
\multicolumn{9}{c}{\s{$_4(A_{25} - 5)$}} \\
\end{tabular} \hspace{0.7cm} \begin{tabular}[b]{|cccc|c|cc|}
\multicolumn{4}{c}{\tiny{$a$}} & \multicolumn{1}{c}{\tiny{$\eta$}} & \multicolumn{2}{c}{\tiny{$b$}} \\ \hline
    $1$ &   $\xi$ & \s{$0$} & \s{$0$} & \s{$0$} & \s{$0$} & \s{$0$} \\
\s{$0$} &     $1$ & \s{$0$} & \s{$0$} & \s{$0$} &     $1$ &   $\xi$ \\
\s{$0$} & \s{$0$} &     $1$ &   $\xi$ & \s{$0$} & \s{$0$} &     $1$ \\
\s{$0$} & \s{$0$} & \s{$0$} &     $1$ &     $1$ & \s{$0$} & \s{$0$} \\ \hline
\multicolumn{7}{c}{\s{$_4(A_{25}^* - 5)$}} \\
\end{tabular}
\caption{Example of extended corepresentations, $_4(A_{25} - 5) =\ _4(K_6 - 4)^{\eta}$ and $_4(A_{25}^* - 5) = \ _4(K_6 - 4)_{\eta}$.}\label{Example of corepresentation extended}
\end{table}

\begin{proof}[Proof of Theorem C1]
We recall the main facts of that classification. Let $U$ be an indecomposable corepresentation of $K_6$. There are considered three cases. When $\US{U}_a \not= U_0$ or $\US{U}_b \not= U_0$, then $U$ may be considered as no sincere indecomposable corepresentation of $A_{25}$ and then obtained by integration $\seven$ from another one of $A_{25}$, in the process is also obtained the classification of $A_{25}$, see proposition \ref{Classification of coreps of A25} and table \ref{List of coreps of A25}. The second case is when $\US{U}_a = \US{U}_b = U_0$ and $\UI{U}_a = \UI{U}_b = 0$, then by matrix considerations the problem corresponds to the matrix pencil problem of Kronecker over the field $\F$ so, it is obtained the indecomposable corepresentation type $_n(K_6 - 5)$ and the series $S(X) = \ _n(K_6 - 6)(X)$, where $X$ is an indecomposable matrix Frobenius block (Rational block in other terminology) over $\F$. The last case is when $\US{U}_a = \US{U}_b = U_0$ and, $\UI{U}_a \not= 0$ or $\UI{U}_b \not= 0$, this is the dual case of the first one, then $U$ may be obtained by duality, and classification for $A_{25}^*$ is also obtained, then it holds the following proposition.
\end{proof}

\begin{table}[hbtp]\renewcommand{\arraystretch}{1}\renewcommand{\tabcolsep}{1mm}
\centering
\begin{tabular}[b]{|c|c|}
\multicolumn{1}{c}{\tiny{$a$}} & \multicolumn{1}{c}{\tiny{$b$}} \\ \hline
$I_n$ & $I_n + \xi J_n^+(0)$ \\ \hline
    \multicolumn{2}{c}{\s{$_n(K_6 - 5)$}} \\
\end{tabular} \hspace{0.7cm} \begin{tabular}[b]{|c|c|}
\multicolumn{1}{c}{\tiny{$a$}} & \multicolumn{1}{c}{\tiny{$b$}} \\ \hline
$I_n$ & $\xi I_n + X$ \\ \hline
    \multicolumn{2}{c}{\s{$_n(K_6 - 6)(X)$}} \\
\end{tabular}
\caption{$_n(K_6 - 5)$ and the series of indecomposable corepresentations $S(X) = \ _n(K_6 - 6)(X)$, where $X$ is a canonical matrix Frobenius block over $\F$.}\label{series of coreps of K6}
\end{table}

\begin{prop}\label{Classification of coreps of A25}
Each of the posets $A_{25} = \{ a ; \: b\lhd \eta \}$ and $A_{25}^* = \{ a ; \: \eta\lhd b \}$ has 5 types of indecomposable corepresentations sincere at the points $\eta$, they are obtained in the form of table \ref{List of coreps of A25}.
\end{prop}

\begin{table}[hbtp]
\centering
\renewcommand{\arraystretch}{1.2}
\renewcommand{\tabcolsep}{1.5mm}
\setlength{\doublerulesep}{0.14pt}
\setlength{\arrayrulewidth}{0.2pt}
\scalebox{0.75}{
\begin{tabular}{|rcl|rcl|rcl|}\hline
$     _n(K_6 - 1^*) $ & = & $      \Int_n(A_{25} - 1) $ & $_n(A_{25} - 1)     $ & = & $ \Int_n(K_6 - 1^*)       $ & $_n(A_{25}^* - 1)     $ & = & $_n(K_6 - 1)_{\eta}        $ \\
$_{2n+1}(K_6 - 2^*) $ & = & $   \Int_{2n}(A_{25} - 2) $ & $_{2n}(A_{25} - 2)  $ & = & $ \Int_{2n-1}(A_{25} - 3) $ & $_{2n}(A_{25}^* - 2)  $ & = & $_{2n}(A_{25} - 2)^*       $ \\
$  _{2n+1}(K_6 - 3) $ & = & $   \Int_{2n}(A_{25} - 5) $ & $_{2n+1}(A_{25} - 3)$ & = & $ \Int_{2n+1}(K_6 - 2^*)  $ & $_{2n+1}(A_{25}^* - 3)$ & = & $_{2n+1}(K_6 - 2)_{\eta}   $ \\
$    _{2n}(K_6 - 4) $ & = & $ \Int_{2n-1}(A_{25} - 4) $ & $_{2n+1}(A_{25} - 4)$ & = & $ \Int_{2n+1}(K_6 - 3)    $ & $_{2n+1}(A_{25}^* - 4)$ & = & $_{2n+1}(K_6 - 3^*)_{\eta} $ \\
                      &   &                              & $_{2n}(A_{25} - 5)  $ & = & $ \Int_{2n}(K_6 - 4)      $ & $_{2n}(A_{25}^* - 5)  $ & = & $_{2n}(K_6 - 4)_{\eta}     $ \\ \hline
\end{tabular}   }
\caption{Classification of indecomposable corepresentations of $A_{25}$, $A_{25}^*$ and part of $K_6$. Initial corepresentations are $_1(A_{25} - 1) = (F_{14}-A)$, $_2(A_{25} - 2) = (F_{14}-B)$ and $_1(A_{25} - 4) = (F_{14}-B)$. Indecomposable corepresentations $(K_6 - i) = (K_6 - i^*)^*$ for $i = 1,\ 2$ are obtained by duality, remaining corepresentations are self-dual, it means that $(K_6 - i) = (K_6 - i^*)$, for $i= 3, 4$ . In all cases $n\ge 1$.}\label{List of coreps of A25}
\end{table}

Notice that $ Int\ _n(K_6 - i)$ coincides with the extended corepresentation $_n(K_6 - i)^{\eta}$ for $i \in \{ 1^*, 2^*, 3, 4 \}$, see an example in table \ref{Example of corepresentation extended}, and $(A_{25} - i)^* \not= (A_{25}^* - i)$ for $i\not= 2$, for those cases apply lemma \ref{dualsincere}.

Each of the posets $A_{28}= \{ a \lhd  \zeta ; \: b  \lhd \eta \}$, $A_{29} = \{a ;\: b  \lhd  \eta  \lhd  \theta\}$, $A_{34} = \{ a ; \: \eta  \lhd  b  \lhd  \theta \}$ and $A_{45} = \{ a  \lhd  \zeta  \rhd  \eta  \lhd  b \}$ where $a$ and $b$ are the only weak points, has a $\seven$-suitable pair for differentiation $\seven$, (see section 4 in \cite{Rod-Zav-07}) then by integrating all indecomposable corepresentations of the respective derived sets we obtain its classification. For the poset $A_{33}=\{ a \lhd \zeta ; \: \eta \lhd  b \}$, we apply $(\zeta , \eta)$-\textbf{completion} (for ordinary posets), in the sense of \cite{Zav-05-TP}, i.e. we add the relation $\zeta \rhd  \eta$, so that, $A_{33}$ and $A_{45}$ have the same type of sincere indecomposable corepresentations, see lemma 5.2 in \cite{Zav-05-TP}. Notice that each of the posets $A_{33}, \, A_{34}, \, A_{45}$ is self-dual. it follows the following proposition.

\begin{prop}\label{Classification of A28, A29, A33, A34 and A45}
Each of the posets $A_{28} $, $A_{29}$, $A_{33}$, $A_{34}$, $A_{45}$ and their dual posets has 1 type of indecomposable corepresentations sincere at all its strong points having the following forms:
\end{prop}

\begin{table}[htbp]
\centering
\renewcommand{\arraystretch}{1.2}
\renewcommand{\tabcolsep}{1.5mm}
\setlength{\doublerulesep}{0.14pt}
\setlength{\arrayrulewidth}{0.2pt}
\begin{tabular}{|rcl|rcl|}\hline
$_{2n-1}(A_{28})$ & = &  $ _{2n-1}(A_{25} - 3)^{\zeta}  $ & $ _{2n-1}(A_{28}^*)$ & =  & $_{2n-1}(A_{25}^* - 3)_{\zeta}$ \\
$_{2n}(A_{29})  $ & = &  $ _{2n}(A_{25} - 2)^{\theta}   $ & $ _{2n}(A_{29}^*)  $ & =  & $_{2n}(A_{25}^* - 2)_{\theta} $ \\
$_{2n-1}(A_{33})$ & = &  $ _{2n-1}(A^*_{25} - 4)^{\zeta}$ & $ _{2n-1}(A_{45})  $ & =  & $_{2n-1}(A^*_{25} - 4)^{\zeta}$ \\
$_{2n}(A_{34})  $ & = &  $ _{2n}(A^*_{25} - 5)^{\theta} $ & $                  $ &    &                                 \\ \hline
\end{tabular}
\caption{Classification of indecomposable corepresentations of the posets $A_{28}$, $A_{29}$, $A_{33}$, $A_{34}$, and their duals.}\label{List of coreps of A28, A29, A33, A34 and 45}
\end{table}

The classifications of $A_{38} = \{ a \prec q \succ b \}$ and $A_{39}=\{ a \prec  q  \succ  b  \lhd  \theta \}$ where $\theta$ is the only strong point, are obtained in the same way than classifications of $K_6$ and $A_{25}$. Let $U$ be an indecomposable corepresentation of $A_{38}$. There are considered two cases. $\US{U}_a \not= U_0$ or $\US{U}_b \not= U_0$. Suppose $\US{U}_b \not= U_0$ (other case is analogous), then $U$ may be extended to a corepresentation $V$ of $A_{39}$ which is not sincere at $\theta$, namely $V_0 = U_0$, $V_x = U_x$ for $x\not= \theta$ and $V_{\theta} = \US{U}_b$, since $A_{39}$ has a $\seven $-suitable pair $(a, \theta)$ then $U$ is obtained by integration of indecomposable corepresentations of $A_{39}'$, see tables \ref {List of coreps of A38} and \ref{List of coreps of A39}. Second case, $\US{U}_a = U_0$ and $\US{U}_b = U_0$, then by special reduction methods that involves corepresentations of $K_6$ satisfying that condition, it follows that $U =\ _n(K_6 - 5)^{q}$. Classification of indecomposable corepresentations of $A_{38}$ was obtained in \cite{Rod-10} by other special reductions which lead to a matrix problem with respect to representations of 2-equipped posets. For $A_{38}^* = \{ a  \succ  q  \prec  b \}$ and $A_{39}^*=\{ a  \succ  q  \prec  b  \rhd  \theta \}$ is used duality and lemma \ref{dualsincere}.

\begin{prop}\label{Classification of A38}
Each of the posets $A_{38}$ and $A_{38}^* $ has 4 types of indecomposable corepresentations sincere at the point $q$ (up to automorphisms of equipped posets), having the following matrix form
\begin{table}[hbtp]
\centering
\renewcommand{\arraystretch}{1.2}
\renewcommand{\tabcolsep}{1.5mm}
\setlength{\doublerulesep}{0.14pt}
\setlength{\arrayrulewidth}{0.2pt}
\scalebox{.75}{
\begin{tabular}{|rcl|rcl|rcl|}\hline
$_{1}(A_{38} - 1)            $ & = & $ (F_{13}-A)  $ &$        _  {2n+1}(A_{38} - 1)$ & = & $ \Int_{2n}(A_{39} - 4)$ & $_{n}(A^*_{38} - 1)   $  & = & $ _{n}(K_6 - \widetilde{1})_q $\\
$_{1}(A_{38} - \widetilde{2})$ & = & $ (F_{17})    $ &$   _{n+1}(A_{38} - \tilde{2})$ & = & $ \Int_{n}(A_{39} - 2)$ & $_{n}(A^*_{38} - 2)   $  & = & $ _{n}(K_6 - \widetilde{1})_q $\\
$_{1}(A_{38} - 3)            $ & = & $ (K_6 - 5)^q $ &$             _{n}(A_{38} - 3)$ & = & $        _{n}(K_6 - 5)^q$ & $_{n}(A^*_{38} - 3)   $  & = & $ _{n}(K_6 -5)_q              $\\
$_{1}(A_{38} - 4)            $ & = & $ (F_{13}-B)  $ &$          _{2n+1}(A_{38} - 4)$ & = & $ \Int_{2n}(A_{39} - 7)$ & $_{2n-1}(A^*_{38} - 4)$  & = & $ _{2n-1}(K_6 - 2)_q          $\\ \hline
\end{tabular}  }
\caption{Classification of indecomposable corepresentations of $A_{38}$. In all cases $n\ge 1$.}\label{List of coreps of A38}
\end{table}
\end{prop}

\begin{prop}\label{Classification of A39}
Each of the posets $A_{39}$ and $A^*_{39}$ has 10 types of indecomposable corepresentations sincere at the points $q$ and $\theta$, having the following matrix form
\begin{table}[hbtp]
\centering
\renewcommand{\arraystretch}{1.2}
\renewcommand{\tabcolsep}{1.5mm}
\setlength{\doublerulesep}{0.14pt}
\setlength{\arrayrulewidth}{0.2pt}
\begin{tabular}{|rcl|rcl|}\hline
$_2(A_{39}-1) $ & = & $(F_{15} - A)$  & $  _{n+1}(A_{39}-1) $ & = &  $  \Int  _n(A_{39} - 1)            $ \\
$_1(A_{39}-2) $ & = & $(F_{15} - B)$  & $  _n(A_{39}-2)     $ & = &  $  \Int  _n(A_{38} - \widetilde{2})$ \\
$_1(A_{39}-3) $ & = & $(F_{14} - A)$  & $  _n(A_{39}-3)     $ & = &  $  \Int  _n(A_{38} - 1)            $ \\
$_2(A_{39}-4) $ & = & $(F_{15} - C)$  & $  _{n+1}(A_{39}-4) $ & = &  $  \Int  _n(A_{39} - 3)            $ \\
$_2(A_{39}-5) $ & = & $(F_{15} - D)$  & $  _{n+1}(A_{39}-5) $ & = &  $  \Int  _n(A_{39} - 5)            $ \\
$_1(A_{39}-6) $ & = & $(F_{14} - B)$  & $  _{2n-1}(A_{39}-6)$ & = &  $  \Int  _{2n-1}(A_{38} - 4)       $ \\
$_2(A_{39}-7) $ & = & $(F_{15} - E)$  & $ _{2n}(A_{39}-7)   $ & = &  $  \Int  _{2n-1}(A_{39} - 6)       $ \\
$_2(A_{39}-8) $ & = & $(F_{14} - C)$  & $ _{2n}(A_{39}-8)   $ & = &  $  \Int  _{2n-1}(A_{39} - 10)      $ \\
$_3(A_{39}-9) $ & = & $(F_{15} - G)$  & $ _{2n+1}(A_{39}-9) $ & = &  $  \Int  _{2n}(A_{39} - 8)         $ \\
$_3(A_{39}-10)$ & = & $(F_{15} - F)$  & $ _{2n+1}(A_{39}-10)$ & = &  $  \Int  _{2n-1}(A_{39} - 9)       $ \\ \hline
\end{tabular}
\caption{Classification of indecomposable corepresentations of $A_{39}$. In all cases $n\ge 2$.}\label{List of coreps of A39}
\end{table}
\end{prop}

Indecomposable corepresentations of the poset $A_{39}^*$ are obtained in the following way, $(A_{39^*} - i) = (A_{39} - i)^*$ provided that $i = 1,4,5,8,9,10$, for $(A_{39}^* - 7)$ apply the lemma \ref{dualsincere}, remaining corepresentations are $(A_{39}^* - 2) = (A_{38}^* - 2)_{\theta}$; $(A_{39}^* - 3) = (A_{38}^* - 1)_{\theta}$; $(A_{39}^* - 6) = (A_{38}^* - 4)_{\theta}$.

The classification of indecomposable corepresentations of $A_{41} = \{ p \prec \{a, b\} \prec q  \}$ where all points are weak, was obtained in \cite{Rod-10} by applying special matrix reductions which lead to the problem of classifying representations of $K_8$. Notice that $A_{46} = \{ p \prec \{a, b\} \prec q;\: p\lhd q \}$ is obtained by applying $(p, q)$-completion to $A_{41}$ (see \cite{Rod-Zav-07}), i.e. by strengthening the relation $p\prec q$, thus $A_{41}$ and $A_{46}$ have the same indecomposable corepresentations except for a trivial one.

\begin{prop}\label{Classification of A41 and A46}
Each of the posets $A_{41}$, and $A_{46}$ has 1 type, up to automorphism of 2-equipped posets, of indecomposable corepresentations sincere at points $p$ and $q$ having the following matrix forms
\begin{align*}
_{n}(A_{41}) =\, _{n}(A_{46}) =\,  _{n}(K_6 - 5)_p^q
\end{align*} for $n\ge 2$
\end{prop}

\begin{prop}\label{A42coreps}
Each of the posets $A_{42} = \{ \{ a ,b \} \prec q \prec r \}$ and $A^*_{42} = \{ r \succ q \succ \{ a ,b \} \}$ has 1 type, up to automorphism of equipped posets, of indecomposable corepresentations sincere at the point $d$, having the following matrix forms
\begin{align*}
_{1}(A_{42}) &= F_{17} & _{n}(A_{42}) &= \, _{n}(A_{38} - 1)^{r} \\
_{1}(A_{42})^* &= F_{17} & \, _{n}(A_{42}^*) &= \, _{n}(A_{42}-0)^*
\end{align*}where $_{n}(A_{42} - 0) := \, _{n}(A_{38} - 1)$ with $\supp n(A_{42}-0) = \{ a, b, q \} $ and $n\ge 2$
\end{prop}

\begin{proof}
Let $U$ be a sincere indecomposable corepresentation of $A_{42}$, consider the restriction $V = U|_{A_{38}}$ where $A_{38} = \{ a \prec q \succ b \}$ and decompose it in a direct sum $V = \bigoplus_i V^i$. Since $U_a + U_b \not= U_0$, it follows from Teorema 3.1 \cite{Rod-10} that $\UI{U}_a = \UI{U}_b = 0$ and also $\UI{V}_a = \UI{V}_b = 0$, therefore each summand satisfies the condition $\UI{V}_a^i = \UI{V}_b^i = 0$, further $\codim_{\F} V_q^i \not= 0$, otherwise $V^i$ would be a direct summand of $U$. So, either each summand $V^i$ is sincere at the point $q$ and then is of $(A_{38} - 1)$ type or each summand is not sincere at $q$ and then is of the form $(K_6 - 1^*), \ (K_6 - \overline{1^*}), \ (K_6 - 2^*)$ or $(K_6 - 5)$.

Consider the matrix realization $M = M_U$, and reduce the vertical stripes corresponding to the points $a$, $b$, $q$ as a direct sum of powers of indecomposable corepresentations as mentioned above, grouped according to its dimension. On the vertical stripe $r$ is obtained a matrix problem concerning about classifying indecomposable representations of a garland of the form (\ref{guirnalda orden}).
\begin{equation}\label{guirnalda orden}
\begin{split}
&a_0 < b_0 < \{ c_0, \tilde{c}_0 \} < b_1 < a_1 < \cdots  \\ &< a_n< b_{2n} <\{ c_n, \tilde{c}_n\} < b_{2n+1} < a_{n+1} < d_n <  \cdots < d_1
\end{split}
\end{equation}

More precisely, the matrix blocks in the stripes of the strong points\footnote{In the sense of representations of 2-equipped posets, see \cite{Zab-Zav-99}, \cite{Zav-03-TEP}}(weak points) $b_i$'s, $c_j$'s, $\tilde{c}_k$'s, $d_l$'s ($a_m$'s) correspond to the corepresentations $_{i+1}(A_{38}-1)^{n_i}$, $_{i+1}(K_6 - 1^*)^{n_j}$, $_{i+1}(K_6 - \tilde{1}^*)^{n_k}$ and $_{i}(K_6 - 5)^{n_l}$ ($_{2i + 1}(K_6-2^*)^{n_m}$) respectively. An example is given in Figure \ref{guirnalda} below.

Then, the problem to classify indecomposable corepresentations of $A_{42}$ is reduced to classify indecomposable representations of a garland (\ref{guirnalda orden}), which is given by the matrix blocks
\begin{equation}\label{classification of reps of a garland}
\renewcommand{\arraystretch}{1.05}\begin{tabular}[c]{|c|}\multicolumn{1}{c}{$x$} \\
\hline 1 \\ \hline \multicolumn{1}{c}{} \\   \end{tabular}\ , \quad
 \begin{tabular}[c]{|c|} \multicolumn{1}{c}{$a_i$} \\
\hline
1 \\
$\xi$ \\ \hline \multicolumn{1}{c}{} \\
\end{tabular}\ ,  \quad
\begin{tabular}[c]{|c|c|} \multicolumn{1}{c}{$c_i$} & \multicolumn{1}{c}{$\tilde{c}_i$} \\
\hline
1 & 1 \\ \hline \multicolumn{1}{c}{} \\
\end{tabular}
\end{equation}
where $x$ is any of the points $a_i$, $b_i$, $c_i$, $\tilde{c}_i$ or $d_i$. By checking exhaustively all corepresentations it is obtained that $_{n}(A_{42}) =\, _{n}(A_{38} - 1)^{r}$ is the only one which is sincere at any point of $A_{42}$. The classification of indecomposable corepresentations of $A_{42}^*$ is obtained by duality, see lemma \ref{dualsincere}.
\end{proof}

\begin{figure}[hbtp]
\begin{picture}(420,250)
\put(9,0){\renewcommand{\arraystretch}{1.1}
\renewcommand{\tabcolsep}{1.45mm}
\setlength{\doublerulesep}{0.14pt}
\setlength{\arrayrulewidth}{0.2pt} \small\normalsize
\begin{tabular}[b]{||ccccccccc||ccccccccc||ccc||c||}
\mc{9}{c}{$a$}&\mc{9}{c}{$b$}&\mc{3}{c}{$q$}&\mc{1}{c}{$r$} \\ \hline
I&\mc{1}{|c}{}&&&&&&&&I&\mc{1}{|c}{}&&&&&&&&&&&$\xi\F$ \\ \cline{1-3} \cline{10-12} \cline{22-22}
&\mc{1}{|c}{I}&\mc{1}{c|}{$0$}&&&&&&&&\mc{1}{|c}{I}&\mc{1}{c|}{$\xi$I}&&&&&&&&&&$0$ \\
&\mc{1}{|c}{$0$}&\mc{1}{c|}{I}&&&&&&&&\mc{1}{|c}{$0$}&\mc{1}{c|}{I}&&&&&&&&&&$\xi\F$ \\ \cline{2-5}  \cline{11-14} \cline{19-19} \cline{22-22}
&&&\mc{1}{|c}{I}&\mc{1}{c|}{$0$}&&&&&&&&\mc{1}{|c}{I}&\mc{1}{c|}{$0$}&&&&&\mc{1}{c|}{$0$}&&&$\xi\F$ \\
&&&\mc{1}{|c}{$0$}&\mc{1}{c|}{I}&&&&&&&&\mc{1}{|c}{$0$}&\mc{1}{c|}{I}&&&&&\mc{1}{c|}{$\xi$I}&&&$0$ \\
&&&\mc{1}{|c}{$0$}&\mc{1}{c|}{$0$}&&&&&&&&\mc{1}{|c}{$\xi$I}&\mc{1}{c|}{I}&&&&&\mc{1}{c|}{$0$}&&&$0$ \\ \cline{4-7} \cline{13-16} \cline{19-19} \cline{22-22}
&&&&\mc{1}{c|}{}&I&$0$&\mc{1}{|c}{}&&&&&&\mc{1}{c|}{}&I&$0$&\mc{1}{|c}{}&&&&&$\xi\F$ \\
&&&&\mc{1}{c|}{}&$0$&\mc{1}{c}{I}&\mc{1}{|c}{}&&&&&&\mc{1}{c|}{}&$0$&I&\mc{1}{|c}{}&&&&&$\xi\F$ \\
&&&&\mc{1}{c|}{}&$0$&$0$&\mc{1}{|c}{}&&&&&&\mc{1}{c|}{}&$\xi$I&I&\mc{1}{|c}{}&&&&&$0$ \\ \cline{6-8} \cline{15-17} \cline{20-20} \cline{22-22}
&&&&&&&\mc{1}{|c|}{I}&&&&&&&&&\mc{1}{|c|}{I}&&&\mc{1}{|c|}{$0$}&&$\xi\F$ \\
&&&&&&&\mc{1}{|c|}{$0$}&&&&&&&&&\mc{1}{|c|}{I}&&&\mc{1}{|c|}{$\xi$I}&&$0$ \\ \cline{8-9} \cline{17-17} \cline{20-20} \cline{22-22}
&&&&&&&\mc{1}{c|}{}&I&&&&&&&&&&&&&$\xi\F$ \\ \cline{9-9} \cline{18-18} \cline{22-22}
&&&&&&&&&&&&&&&&\mc{1}{c|}{}&I&&&&$\xi\F$ \\ \cline{18-18} \cline{21-21} \cline{22-22}
&&&&&&&&&&&&&&&&&&&\mc{1}{c|}{}&I&$\xi\F$ \\  \cline{21-22}
&&&&&&&&&&&&&&&&&&&&&$\G$ \\ \hline
\end{tabular}}
\put(0,229){$d_1$}
\put(0,204){$d_2 \left\{ \begin{array}{c}
\\[15pt]
\end{array}
\right.$}
\put(0,163){$b_2 \left\{ \begin{array}{c}
\\[30pt]
\end{array}
\right.$}
\put(0,116){$a_1 \left\{ \begin{array}{c}
\\[30pt]
\end{array}
\right.$}
\put(0,75){$b_1 \left\{ \begin{array}{c}
\\[15pt]
\end{array}
\right.$}
\put(0,52){$\tilde{c}_0$}
\put(0,37){$c_0$}
\put(0,20){$b_0$}
\put(0,5){$a_0$}
\put(382,229.5){$\circ$}
\put(382,198.5){$\circ$}
\put(382,182.5){$\circ$}
\put(380.35,126.5){$\otimes$}
\put(382,85.5){$\circ$}
\put(398,54){$\circ$}
\put(366,38){$\circ$}
\put(382,22){$\circ$}
\put(380.35,5){$\otimes$}
\put(385,203.5){\line(0,1){27}}
\put(385,187.5){\line(0,1){12}}
\put(385,133.3){\line(0,1){50.2}}
\put(385,90.5){\line(0,1){35.2}}
\put(369.84,43.1){\line(1,3){14.4}}
\put(399.8,58.8){\line(-1,2){13.9}}
\put(386,27){\line(1,2){14}}
\put(383.5,26.5){\line(-1,1){12.9}}
\put(385,11.8){\line(0,1){11.3}}
\end{picture}
\caption{An example of how a problem of classifying indecomposable corepresentations of $A_{42}$ is reduced to a problem of classifying  indecomposable representations of a garland.}\label{guirnalda}
\end{figure}
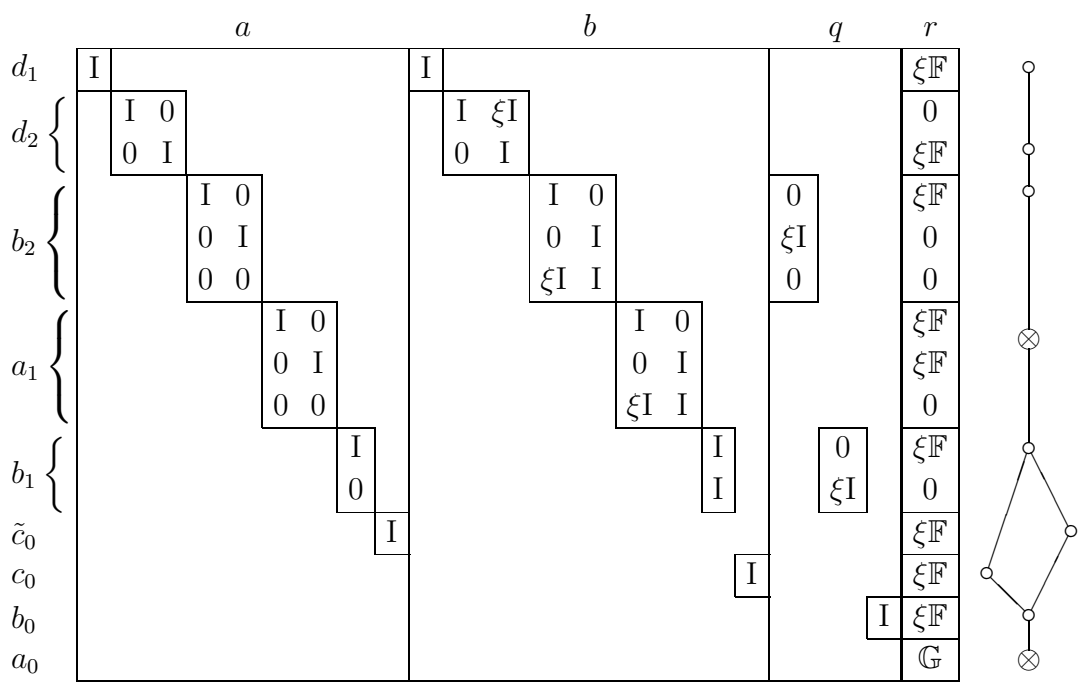

\begin{lemma}\label{properties of Some COREPS}
If $\p $ is one of the sets $K_6$, $A_{25}$, $A_{28}$, $A_{29}$, $A_{38}$, $A_{39}$ or $A_{42}$, then each of indecomposable corepresentation of $\p $ satisfies the following conditions

\begin{enumerate}
\item[(a)] $U_a = U_b = 0 \ \ $ or $\ \ \US{U}_a + \US{U}_b = U_0\: $ for $\: \p \not= A_{39}$

\item[(b)] $U_a = 0 \ \ $ or  $\ \  \begin{cases} \US{U}_a + \US{U}_b = U_0, \ \ \text{for any } \p   \\  \US{U}_a + U_b = U_0, \ \ \text{if } \p \not= A_{38}, A_{39}, A_{42}  \\  \US{U}_a + U_p = U_0, \ \ \text{if } \p = A_{38}, A_{42}  \\  \US{U}_{\theta} + U_p = U_0, \ \ \text{if } \p = A_{39} \end{cases} $
\end{enumerate}
\end{lemma}

It holds by checking exhaustively the classification of indecomposable corepresentations of these posets.

\section{Classification of indecomposable corepresentations of $K_7$}\label{sectionK7}

\begin{proof}[Proof of Theorem C3]
The pair $(a,\theta)$ of the critical 2-equipped poset $K_7=\{ a\prec p \prec q ; \: a\prec q; \: \theta \}$ is suitable for differentiation $\seven$, see Figure \ref{K7 derivado}, then its indecomposable corepresentations are obtained by integrating exhaustively all indecomposable corepresentations displayed in Table \ref{classification of coreps of K7}, even its series of corepresentation comes from $K_6$ series. For this reason the square block $X$ in the series $S(X)$ of $K_7$ (see Table \ref{example of series}) is restricted to be in a Frobenius matrix form over $\F$.
\end{proof}

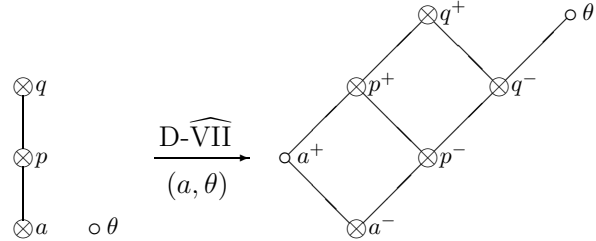
\begin{figure}[h]
\centering
\scalebox{0.9}{
\begin{picture}(250,100)

\multiput(0.35,2)(0,30){3}{$\otimes$}
\multiput(5,8.8)(0,30){2}{\line(0,1){22.3}}
\put(32,2){$\circ$}

\put(62,40){D-$\seven$}
\put(60,35){\vector(1,0){40}}
\put(65,20){$(a,\theta)$}

\put(112,32){$\circ$}
\multiput(140.35,2)(30,30){3}{$\otimes$}
\multiput(140.35,62)(30,30){2}{$\otimes$}
\put(232,92){$\circ$}
\multiput(147.8,7.8)(30,30){2}{\line(1,1){24.4}}
\put(147.8,67.8){\line(1,1){24.4}}
\multiput(147.8,62.2)(30,30){2}{\line(1,-1){24.4}}
\put(116.5,36.5){\line(1,1){25.7}}
\put(116.5,33.5){\line(1,-1){25.7}}
\put(207.8,67.8){\line(1,1){25.7}}

\put(10,63){{\footnotesize $q$}}
\put(10,33){{\footnotesize $p$}}
\put(10,3){{\footnotesize  $a$}}
\put(40,3){{\footnotesize  $\theta$}}

\put(180,93){{\footnotesize $q^+$}}
\put(150,63){{\footnotesize $p^+$}}
\put(120,33){{\footnotesize  $a^+$}}
\put(240,93){{\footnotesize  $\theta$}}
\put(210,63){{\footnotesize $q^-$}}
\put(180,33){{\footnotesize $p^-$}}
\put(150,3){{\footnotesize  $a^-$}}

\end{picture}}
\caption{Differentiation of $K_7$}\label{K7 derivado}
\end{figure}

\begin{table}
\renewcommand{\arraystretch}{1.1}
\renewcommand{\tabcolsep}{1.45mm}
\setlength{\doublerulesep}{0.14pt}
\setlength{\arrayrulewidth}{0.2pt} \small\normalsize
\begin{tabular}{|c|rcl|c|c|rcl|c|}\hline
T  & \multicolumn{3}{|c|}{corepresentation} & f & T & \multicolumn{3}{|c|}{corepresentation} & f \\ \hline
1  &$(K_7 - 1) $ & = & $\Int (K_6 - 1)             $ &  1 &14  &$(K_7 - 14)$ & = &$\Int (A_{25} - 4)  $  &  2 \\
2  &$(K_7 - 2) $ & = & $\Int (K_6 - \widetilde{1}) $ &  1 &15  &$(K_7 - 15)$ & = &$\Int (A_{33})      $  &  2 \\
3  &$(K_7 - 3) $ & = & $\Int (K_6 - 2)             $ &  2 &16  &$(K_7 - 16)$ & = &$\Int (A_{39} - 1)  $  &  1 \\
4  &$(K_7 - 4) $ & = & $\Int (K_6 - 3)             $ &  2 &17  &$(K_7 - 17)$ & = &$\Int (A_{39} - 2)  $  &  1 \\
5  &$(K_7 - 5) $ & = & $\Int (K_6 - 4)             $ &  2 &18  &$(K_7 - 18)$ & = &$\Int (A_{39} - 3)  $  &  1 \\
6  &$(K_7 - 6) $ & = & $\Int (A_{38} - 1)          $ &  1 &19  &$(K_7 - 19)$ & = &$\Int (A_{39} - 4)  $  &  1 \\
7  &$(K_7 - 7) $ & = & $\Int (A_{38} - 2)          $ &  1 &20  &$(K_7 - 20)$ & = &$\Int (A_{39} - 5)  $  &  1 \\
8  &$(K_7 - 8) $ & = & $\Int (A_{38}-\widetilde{2})$ &  1 &21  &$(K_7 - 21)$ & = &$\Int (A_{39} - 6)  $  &  2 \\
9  &$(K_7 - 9) $ & = & $\Int (A_{38} - 3)          $ &  1 &22  &$(K_7 - 22)$ & = &$\Int (A_{39} - 7)  $  &  2 \\
10 &$(K_7 - 10)$ & = & $\Int (A_{38} - 4)          $ &  2 &23  &$(K_7 - 23)$ & = &$\Int (A_{39} - 8)  $  &  2 \\
11 &$(K_7 - 11)$ & = & $\Int (A_{25} - 1)          $ &  1 &24  &$(K_7 - 24)$ & = &$\Int (A_{39} - 9)  $  &  2 \\
12 &$(K_7 - 12)$ & = & $\Int (A_{25} - 2)          $ &  2 &25  &$(K_7 - 25)$ & = &$\Int (A_{39} - 10) $  &  2 \\
13 &$(K_7 - 13)$ & = & $\Int (A_{25} - 3)          $ &  2 &    &             &   &                         &    \\ \hline
\end{tabular}
\caption{Classification of indecomposables corepresentations of the critical poset $K_7$. Indecomposable corepresentations types $(K_7 - 5)$ and $(K_7 - 15)$ are self-dual. Remaining corepresentations are definided as follows $(K_7 - i^*) = (K_7 - i)^*$.
}\label{classification of coreps of K7}
\end{table}

Notice that an one-parameter 2-equipped poset $\p $ containing a critical subset $K_7$ has the following form
\begin{equation}\label{K7 subset}
\p  = K_7 + (K_7)^{\blacktriangledown} + (K_7)^{\blacktriangle}.
\end{equation}
Furthermore, given $x\in (K_7)^{\blacktriangledown} \subset\p$, it satisfies one of the following conditions: \textbf{(1)} $q\lhd x$; \textbf{(2)} $\{ p, \theta \} \lhd x$; \textbf{(3)} $q \prec x\: $ and $\: \theta \lhd x$; \textbf{(4)} $q \prec x\:$ and $\: p \lhd x$. If $x\in (K_7)_{\blacktriangle}$ it satisfies one of the dual conditions: \textbf{(1')} $a\rhd x$; \textbf{(2')} $\{ p, \theta \} \rhd x$; \textbf{(3')} $x\prec a\:\: $ and $\:\: \theta \rhd x$; \textbf{(4')} $x\prec a\:\: $ and $\:\:p \rhd x$.

Let $W$ be a corepresentation of $\p$ as in (\ref{K7 subset}), its codimension vector is denoted by $r = r_W = (r_1, r_2, r_3, r_4)$ where $r_1 = \codim_{\G} \US{W}_q$; $r_2 = \codim_{\G} \US{W}_p + W_{\theta}$; $r_3 = \codim_{\F} W_q + W_{\theta}$; $r_4 = \codim_{\F} W_q + \US{W}_p$. For dual case is defined $r^* = r^*_W = (r_1^*, r_2^*, r_3^*, r_4^*)$ where $r_1^* = \dim_{\G} \UI{W}_a$; $r_2^* = \dim_{\G} \UI{W}_p \cap W_{\theta}$; $r_3^* = \dim_{\F} W_a \cap W_{\theta}$; $r_4^* = \dim_{\F} W_a \cap \UI{W}_p$.

By checking exhaustively all indecomposable corepresentations of $K_7$ (see Table \ref{classification of coreps of K7}) it follows that indecomposable corepresentations $W$ of $K_7$ such that $r_W\not = 0$ are the listed ones in table \ref{K7 coreps r no 0} (denoted by $W_i$ with $i\in\{ 1,\ldots , 9 \}$). Notice that $r^*_{W_i} = 0$ for all $i$.

\begin{table}[hbtp]
\renewcommand{\tabcolsep}{2mm}
\renewcommand{\arraystretch}{1.1}
\begin{tabular}{llll}
$W_{1}\cong \ _1(K_7 - 2^*)$,& $r_{W_1} = (0,0,0,1)$; &  $W_{6}\cong \ _1(K_7 - 12)$, & $r_{W_6} = (0,1,1,1)$;\\
$W_{2}\cong \ _2(K_7 - 8)$,  & $r_{W_2} = (0,0,0,1)$; &  $W_{7}\cong \ _1(K_7 - 15)$, & $r_{W_7} = (0,2,0,0)$;\\
$W_{3}\cong \ _1(K_7 - 9)$,  & $r_{W_3} = (0,0,1,0)$; &  $W_{8}\cong \ _2(K_7 - 21)$, & $r_{W_8} = (0,1,0,2)$;\\
$W_{4}\cong \ _2(K_7 - 10)$, & $r_{W_4} = (0,0,0,1)$; &  $W_{9}\cong \ _1(K_7 - 22)$, & $r_{W_9} = (1,1,2,2)$;\\
$W_{5}\cong \ _1(K_7 - 11)$, & $r_{W_5} = (0,0,1,0)$; &                               &                       \\
\end{tabular}
\caption{Corepresentations of $K_7$ with $r\not=0$.}\label{K7 coreps r no 0}
\end{table}

\begin{lemma}
Let $\p$ be a sincere one-parameter 2-equipped poset containing the subposet $K_7$, then $(K_7)^{\blacktriangledown} = \emptyset$\hspace{2mm} or \hspace{2mm}$(K_7)_{\blacktriangle}= \emptyset$.
\end{lemma}

\begin{proof}
Let $U \in \ind \p $ sincere at $x \in (K_7)^{\blacktriangledown}$, and $V = U|_{K_7}$ be a restricted corepresentation to $K_7$, then decompose it in a direct sum $V = \bigoplus_j V^j$. Since $U$ is sincere at $x$, for every summand $V^j \simeq W_i$ ($i\in \{1, \ldots , 9\}$), thus $U$ is not sincere at any point in $(K_7)_{\blacktriangle}$ because of $r^*_W = 0$ for $W=W_i$. Analogously, $U$ is not sincere in $(K_7)^{\blacktriangledown}$ provided that $U\in\ind \p$ is to be sincere at some point in $(K_7)_{\blacktriangle}$.
\end{proof}

\begin{lemma}\label{K7 the one}
Let $\p$ be a sincere one-parameter 2-equipped poset containing a subset $K_7$. Then $\p = K_7$.
\end{lemma}

\begin{proof}

One can assume that $\p = K_7 + (K_7)^{\blacktriangledown}$, the case $\p = K_7 + (K_7)_{\blacktriangle}$ is analogous. The proof uses induction on $|(K_7) ^{\blacktriangledown}| \ge 1$. Let $x\in (K_7)^{\blacktriangledown}\cap \Max \p$\hspace{2mm} and \hspace{2mm}$U\in\ind\p$ be a sincere indecomposable corepresentation with $d = \dimc U$. The restriction $V=U|_{\p\setminus x}$ is considered as direct sum of indecomposable  corepresentations $V =\bigoplus_i V^i$ and its support is denoted by $S^i = \supp V^i$. In the base case $|(K_7) ^{\blacktriangledown}| = 1$ it holds that $S^i\subset K_7$ where $V_i \simeq W$ with $r_W \not = 0$, it means that each summand corresponds to some among the corepresentations in Table \ref{K7 coreps r no 0}, other case $U$ were not sincere at $x$. When $|(K_7) ^{\blacktriangledown}| > 1$ happens that $S^i \subset K_7$ as before, or $S^i\not\subset K_7$, so $S^i$ have to be a finite poset $F_{13}, \ldots , F_{18}$ (whether it contains at least a weak point) or an ordinary finite poset un poset, other case were not of one-parameter type.

The study of corepresentation $U$ is reduced to the following four cases: \textbf{(1)} $q\lhd x$. When $S^i\subset K_7$ each summand $V^i\simeq W$ satisfies $r_1\not= 0$, $W_9$ is the unique summand that does it, but $\dim_{\F} W_a = 0$. If $S^i\not\subset K_7$ there exists $j$ such that $a\in S^j$, notice that $S^j\not= F_{13}, \ldots , F_{16}, F_{18}$ (because it have to contain at least another strong point $z$, and there are two options: (a) $q\lhd z$; (b) $p \lhd z\: $ and $\: \theta \lhd z$; in both of them $a < z$). If $S^j = F_{17}$ $\codim_{\G}V^i_a = 0$. It follows that any summand $V^i\simeq W$ with $r_1 \not= 0$ is not sincere at $a$.

\textbf{(2)} $\{ p, \theta \}\lhd x$. When $S^i\subset K_7$ each summand $V^i\simeq W$ satisfies $r_2\not= 0$, thus is one among $W_6 , \dots , W_9$, any case $dim_{\F} W_a = 0$. If $S^i\not\subset K_7$ there exists $j$ such that $a\in S^j$, for the same reasons of case (2) $S^j\not= F_{13}, \ldots , F_{18}$. It follows that each summand of this form is not sincere at $a$.

\textbf{(3)} $q\prec x\:$ and $\:\theta\lhd x$ (and also $p\prec x$, or it is the second case). When $S^i\subset K_7$ each summand $V^i\simeq W$ satisfies $r_3\not= 0$, summands $W_3 , W_5, W_6 , W_9$ are the only ones which do it, but $\dim_{\G} W_{\theta} = 0$. If $S^i\not\subset K_7$ and $S^i$ has a non trivial equipment, there exists $j$ such that $z, \ \theta \in S^j$ for some weak point $z$, there are there variants: (a) $ q \prec z$ and $\theta \lhd z$, but any finite poset do not satisfy $\theta \lhd z$; (b) $q \prec z\: $ and $\: p\lhd z$, also $z < x$ (to avoid that $K_6 = \{ z, \ x \}\subset \p)$ but then $p\lhd x$ is the second case; (c) $q \lhd z$, for the same reasons implies $p \lhd x$. If $S^j$ is an ordinary finite poset should have at least another strong point $\varphi$ and two cases happen: (a) $p\lhd \varphi \:$ and $\: \theta \lhd \varphi$, and $\varphi \parallel x$ (to avoid that $p\lhd x$), it means that $S^j$ is a chain and then $S^j = F_1\subset K_7$, contradiction; (b) $q\lhd \varphi \:$, $\: \varphi \parallel x$ further $\varphi \parallel \theta$ (or it would be latest case), but then $K_8 = \{ \varphi, \ \theta, \ x \}\subset \p$, contradiction. It is concluded that every summand $V^i\simeq W$ that satisfy $r_3 \not= 0$ are not sincere at $\theta$.

\textbf{(4)} $q\prec x\:$ and $\: p\lhd x$ (and also $x\Inc \theta$, other case it would be third case). There is no point $q < z < x$, due to it implies $q\lhd x\: $ or $\: \theta \lhd x$. If there exist another maximal points, cases (1), (2), (3) apply. Therefore, it is assumed that $\p = K_7 + x$ and $S^i\subset K_7$. Since $r_4\not= 0$, the summands $W \simeq V^i$ are isomorphic to some among $W_1, W_2, W_4, W_6, W_8, W_9$.

Consider the matrix realization $M_U$ of $U$, after reducing the restricted corepresentation $M_V$, a mixed type matrix problem is obtained over the vertical strip $x$, in which are admitted $\G$-elementary column transformations, that strip is separated into six horizontal stripes given by the chain $ 9 < 6 < 8 < 1 < 2 < 4 $ (numbered according to corresponding corepresentation). The stripe(s) $9$ admits ($1,2,4,6,8$ admit) $\G$-elementary ($\F$-elementary) row transformations in such a way that the matrix problem on the vertical stripe $x$ corresponds to a problem on classification of indecomposable representations of a chain (see Figure \ref{K7 y chain}).

The sincere subsets of that chain are singleton sets, it means that just one of $W_1, W_2, W_4, W_6, W_8, W_9$ is the indecomposable corepresentation of $\p\setminus x$, therefore $d_a = 0\:$ or $\: d_b = 0$ any case it contradicts the initial assumption that $U$ is sincere. It follows the lemma. \end{proof}

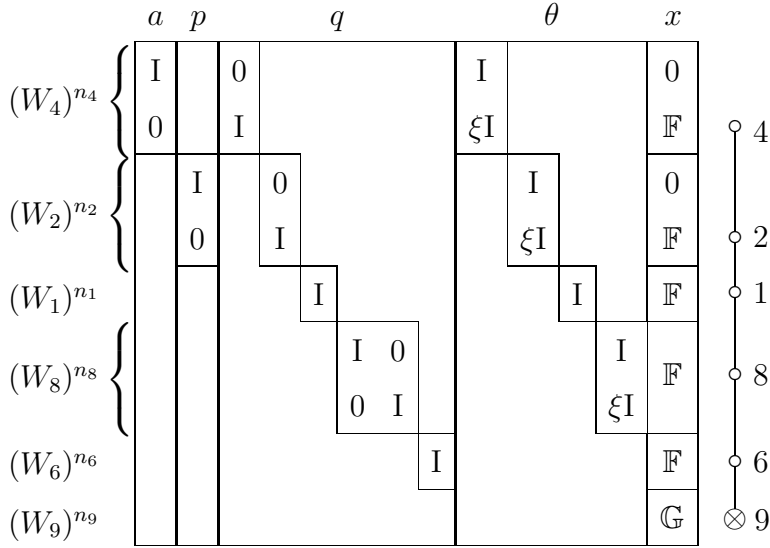
\begin{figure}[hbtp]
\centering
\begin{picture}(450,200)
\put(32,0){\renewcommand{\arraystretch}{1.46}
\renewcommand{\tabcolsep}{1.65mm}
\setlength{\doublerulesep}{0.14pt}
\setlength{\arrayrulewidth}{0.2pt} \small\normalsize
\begin{tabular}[b]{||c||c||cccccc||cccc||c||}
\mc{1}{c}{$a$}&\mc{1}{c}{$p$}&\mc{6}{c}{$q$}&\mc{4}{c}{$\theta$}&\mc{1}{c}{$x$} \\ \hline
  I&   &\mc{1}{c|}{$0$}&&&&&&\mc{1}{c|}{I}&&&&$0$ \\
$0$&   &\mc{1}{c|}{I}&&&&&&\mc{1}{c|}{$\xi$I}&&&&$\F$ \\ \cline{1-2} \cline{3-4} \cline{9-10} \cline{13-13}
   &  I&\mc{1}{c|}{}&\mc{1}{c|}{$0$}&&&&&\mc{1}{c|}{}&\mc{1}{c|}{I}&&&$0$ \\
   &$0$&\mc{1}{c|}{}&\mc{1}{c|}{I}&&&&&\mc{1}{c|}{}&\mc{1}{c|}{$\xi$I}&&&$\F$ \\  \cline{2-2} \cline{4-5} \cline{10-11} \cline{13-13}
   &   &&\mc{1}{c|}{}&\mc{1}{c|}{I}&&&&&\mc{1}{c|}{}&\mc{1}{c|}{I}&&$\F$ \\ \cline{5-7} \cline{11-13}
   &   &&&\mc{1}{c|}{}&I&\mc{1}{c|}{$0$}&&&&\mc{1}{c|}{}&\mc{1}{c|}{I}&\multirow{2}[0]{*}{$\F$} \\
   &   &&&\mc{1}{c|}{}&$0$&\mc{1}{c|}{I}&&&&\mc{1}{c|}{}&\mc{1}{c|}{$\xi$I}& \\ \cline{6-8} \cline{12-13}
   &   &&&&&\mc{1}{c|}{}&I&&&&&$\F$ \\ \cline{8-8} \cline{13-13}
   &&&&&&&&&&&&$\G$ \\ \hline
\end{tabular}}
\put(0,165){$(W_4)^{n_4} \left\{ \begin{array}{c}
\\[26pt]
\end{array}
\right.$}
\put(0,122){$(W_2)^{n_2} \left\{ \begin{array}{c}
\\[26pt]
\end{array}
\right.$}
\put(0,90){$(W_1)^{n_1}$}
\put(0,60){$(W_8)^{n_8} \left\{ \begin{array}{c}
\\[26pt]
\end{array}
\right.$}
\put(0,27){$(W_6)^{n_6}$}
\put(0,5){$(W_9)^{n_9}$}
\put(272,156){$\circ$}
\put(272,114){$\circ$}
\put(272,93){$\circ$}
\put(272,62){$\circ$}
\put(272,29){$\circ$}
\put(270.35,7){$\otimes$}
\put(281.8,153){$4$}
\put(281.8,113){$2$}
\put(281.8,92){$1$}
\put(281.8,61){$8$}
\put(281.8,28){$6$}
\put(282.35,6){$9$}
\put(275,119){\line(0,1){37.7}}
\put(275,98){\line(0,1){16.8}}
\put(275,67){\line(0,1){26.8}}
\put(275,34){\line(0,1){28.8}}
\put(275,14){\line(0,1){16}}
\end{picture}
\caption{Corepresentation of $\p = K_7 + x$, which satisfies (4), and its relationship with representations of a chain. $(W_i)^{n_i} = \bigoplus_{n_i}W_i$. }\label{K7 y chain}
\end{figure}

\section{Classification of one-parameter 2-equipped posets containing $K_8$}\label{sectionK8}

The aim of this section is to give a complete classification of the indecomposable corepresentations of the critical poset $K_8 = \{ \varrho ; \: \sigma ; \: a \}$ where $a$ is the unique weak point (see Figure \ref{critical equipped posets}), and of each sincere one-parameter 2-equipped poset containing it, listed in appendixes  \ref{sincereK8posets}, \ref{specialposets}. These classifications\footnote{ The classification of indecomposable corepresentations of $K_8$ was already obtained in \cite{Rod-10} by using special matrix reductions.} are obtained recursively in terms of differentiation algorithms D-I and completion for ordinary posets \cite{Zav-77, Zav-05-TP}, and D-$\seven$ and completion for 2-equipped posets \cite{Rod-Zav-07}, except the corepresentation series of $K_8$.

\begin{figure}[h]
\scalebox{0.9}{
\begin{picture}(240,90)
\multiput(2,32)(30,0){2}{$\circ$}
\put(60.35,32){$\otimes$}
\put(62,62){$\circ$}
\put(65,39){\line(0,1){24.1}}
\put(33,5){\s{$A_{26}$}}
\put(2,25){\s{$\varrho$}}
\put(32,25){\s{$\sigma$}}
\put(62,25){\s{$a$}}
\put(70,63){\s{$\eta$}}

\put(90,35){\tiny{$(\varrho,\eta)$}}
\put(87,45){\vector(1,0){30}}
\put(85,50){\tiny{D-I-compl.}}

\multiput(132,32)(30,0){2}{$\circ$}
\put(147,62){$\circ$}
\put(190.35,32){$\otimes$}
\put(192,62){$\circ$}
\put(195,39){\line(0,1){24.1}}
\put(136,37){\line(1,2){13}}
\put(164,37){\line(-1,2){13}}
\put(137,36){\line(2,1){55.9}}
\put(166.5,36.5){\line(1,1){26.9}}

\put(163,5){\s{$\overline{A_{26}'}$}}
\put(130,25){\s{$\varrho$}}
\put(163,25){\s{$\sigma^{+}$}}
\put(155,61){\s{$\sigma^{-}$}}
\put(192,25){\s{$a$}}
\put(200,63){\s{$\eta$}}
\end{picture} \begin{picture}(200,90)
\multiput(2,32)(30,0){2}{$\circ$}
\put(60.35,32){$\otimes$}
\put(17,62){$\circ$}
\put(6,37){\line(1,2){13}}
\put(34,37){\line(-1,2){13}}
\put(33,5){\s{$A_{27}$}}
\put(2,25){\s{$\varrho$}}
\put(32,25){\s{$\sigma$}}
\put(62,25){\s{$a$}}
\put(25,63){\s{$\zeta$}}

\put(85,35){\tiny{$(a,\zeta)$}}
\put(82,45){\vector(1,0){30}}
\put(85,50){\tiny{D-$\seven$}}

\multiput(132,32)(30,0){2}{$\circ$}
\put(147,62){$\circ$}
\put(190.35,32){$\otimes$}
\put(192,62){$\circ$}
\put(195,39){\line(0,1){24.1}}
\put(136,37){\line(1,2){13}}
\put(164,37){\line(-1,2){13}}
\put(151.8,63.9){\line(3,-2){40}}

\put(163,5){\s{$A_{27}'$}}
\put(130,25){\s{$\varrho$}}
\put(163,25){\s{$\sigma$}}
\put(159,62){\s{$\zeta$}}
\put(192,25){\s{$a^-$}}
\put(200,63){\s{$a^+$}}
\end{picture}}

\centering
\scalebox{0.9}{
\begin{picture}(240,100)
\multiput(2,62)(60,0){2}{$\circ$}
\put(30.35,62){$\otimes$}
\put(15.35,32){$\otimes$}
\put(21.8,38.6){\line(1,2){11.4}}
\put(18.2,38.6){\line(-1,2){12.2}}
\put(33,5){\s{$A_{40}^*$}}
\put(0,55){\s{$\varrho$}}
\put(18,25){\s{$q$}}
\put(62,55){\s{$\sigma$}}
\put(35,55){\s{$a$}}

\put(85,35){\tiny{$(q,\sigma)$}}
\put(82,45){\vector(1,0){30}}
\put(85,50){\tiny{D-$\seven$}}

\multiput(132,72)(60,0){2}{$\circ$}
\put(147,42){$\circ$}
\put(160.35,72){$\otimes$}
\put(175.35,42){$\otimes$}
\put(160.35,12){$\otimes$}
\put(151,47){\line(1,2){12.2}}
\put(149,47){\line(-1,2){13}}
\put(166.8,18.6){\line(1,2){11.4}}
\put(163.2,18.6){\line(-1,2){12.2}}
\put(181.8,48.6){\line(1,2){12.2}}
\put(178.2,48.6){\line(-1,2){11.6}}
\put(183,5){\s{$(A_{40}^*)^{\prime}$}}
\put(130,65){\s{$\varrho$}}
\put(140,34){\s{$q^+$}}
\put(158,1){\s{$q^-$}}
\put(180,34){\s{$a^-$}}
\put(194,65){\s{$\sigma$}}
\put(159,56){\s{$a^+$}}
\end{picture}}
\caption{Combinatorical description of differentiation I and $\seven$ applied to the posets $A_{26}$, $A_{27}$ and $A_{40}^*$.}\label{Differentiation A26 A27 and A40}
\end{figure}
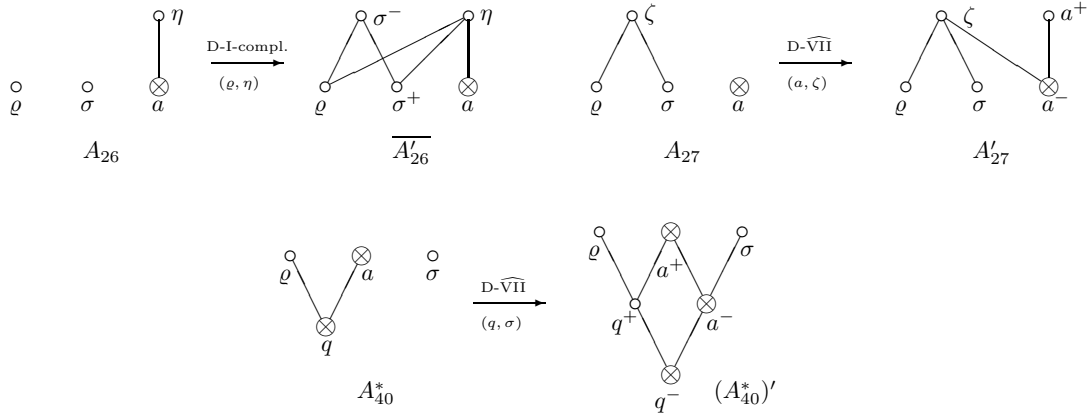

We give a proof of Theorem C3, in which also are obtained classifications of indecomposable corepresentation of $A_{26}$, $A_{27}$, $A_{40}^*$ and their dual posets, see propositions \ref{A26A27coreps} and \ref{A40coreps} below.

\begin{proof}[Proof of Theorem C3]
Let $U$ be an indecomposable corepresentation of $K_8$, three cases can happen:

\textbf{(1)} $U$ satisfies (a) $\codim_{\G} \US{U}_a \not= 0 \; $ or (b) $\; \codim_{\G} \US{U}_{\varrho} + \US{U}_{\sigma} \not= 0\;$ or (c) $\; \dim_{\F} U_a \cap U_{\varrho} \not= 0\; $ or (d) $\; \dim_{\F} U_a \cap U_{\sigma} \not= 0\; $. If $U$ satisfies (a), (b) or (c) (condition (d) is analogous to (c)), it can be extended to a corepresentation $V$ of $A_{26} = \{ \varrho , \ \sigma , \ a \lhd \eta \}$, $A_{27} = \{ \varrho \lhd \zeta \rhd \sigma , \ a \}$ or $A_{40} = \{ \varrho \rhd q \lhd a, \ \sigma \}$ respectively (where $a$ and $q$ are the only weak points), namely, in case (a) by $V_{\eta} = \US{U}_a$, (b) by $V_{\zeta} = U_{\varrho} + U_{\sigma}$, or (c) by $V_q = U_{\varrho}\cap U_a$. According to the case, $V$ is obtained by integration from an indecomposable corepresentation of the completed derived poset $\overline{(A_{26}^{\prime})}$ with respect to D-I in case (a), of the derived poset $A_{27}^{\prime}$ in case (b) and of the derived poset $A_{40}^{\prime}$ in case (c) with respect to D-$\seven$, see Figure \ref{Differentiation A26 A27 and A40}. Finally $U$ is obtained by restricting $V|_{K_8}$.

\textbf{(2)} $U$ satisfies (a') $\dim_{\G} \UI{U}_a \not= 0 \; $ or (b') $\; \dim_{\G} \UI{U}_{\varrho} \cap \US{U}_{\sigma} \not= 0\;$ or (c') $\; \codim_{\F} U_a + U_{\varrho} \not= 0\; $ or (d') $\; \codim_{\F} U_a + U_{\sigma} \not= 0\; $. It is the dual condition of (1). Any corepresentation $U$ of $K_8$ which satisfies condition (2) may be obtained by duality from some one which satisfies (1).

\textbf{(3)} $U$ satisfies $\codim_{\G} \US{U}_a = \dim_{\G} \UI{U}_a = 0 \; $ and $\; \codim_{\G} \US{U}_{\varrho} + \US{U}_{\sigma} = \dim_{\G} \UI{U}_{\varrho} \cap \US{U}_{\sigma} = 0\;$ and $\; \dim_{\F} U_a \cap U_{\varrho} = \codim_{\F} U_a + U_{\varrho} = 0\;$ and $\; \dim_{\F} U_a \cap U_{\sigma} = \codim_{\F} U_a + U_{\sigma} = 0\;$. Consider a matrix realization of $U$, then it may be turned into a matrix form
$M =$ \renewcommand{\arraystretch}{1.1}\renewcommand{\tabcolsep}{2mm}
\setlength{\doublerulesep}{0.08pt}\begin{tabular}[c]{||c||c||c|c||}
  \multicolumn{1}{c}{$\circ$} & \multicolumn{1}{c}{$\circ$} & \multicolumn{2}{c}{$\otimes$} \\ \hline\hline
  $I_n$&&$\xi I_n$&$I_n$ \\ \hline
  &$I_n$&$A$&$B$ \\ \hline\hline
\end{tabular} where $A$, $B$ are square matrices over $\G$. The problem to find a pair of matrices $(A,B)$, in such a way that $M$ is indecomposable, is reduced to a biquadratic matrix problem\footnote{ This problem was solved in \cite{Zav-07-Kp} in characteristic $\not= 2$, and solved in characteristic $= 2$ in \cite{Zav-08-bqp}}, which is closely related to the semilinear and pseudolinear Kronecker problems.

The series of indecomposable corepresentations $S(X)$, where $X$ is a square matrix over $\G$ of arbitrary order $n\ge 1$, of the critical poset $K_8$ is defined by the matrix (\ref{Series K8(1)})(a) if the characteristic $\car \F =2$ or (\ref{Series K8(1)})(b) if $\car \F \not=2$ and $\F\subset\G$ is a separable extension or (\ref{Series K8(2)}) if $\car \F \not=2$ and $\F\subset\G$ is inseparable. The series describe almost all indecomposable corepresentations of any shown dimension $d$, however $X$ is restricted to a Frobenius Canonical Form block $C(f)$ associated to a non-constant monic polynomial $f(t) = a_0 + a_1t + \cdots + a_{n-1}t^{n-1} + t^n$ over $\G$

If $f(t) = a_0 + a_1t + \cdots + a_{n-1}t^{n-1} + t^n$ is a non-constant monic polynomial over $\G$, its companion matrix of order $n$ is denoted by $C(f) = (c_{ij})$ with $c_{i,i+1} = 1$, $c_{ni} = -a_{i-1}$ and otherwise $c_{ij} = 0$. Let $\G [t, \sigma , \delta ]$ be the skew polinomial ring of right polynomials where $\sigma: \G \mapsto \G$ is an automorphism and $\delta$ is a $\sigma$-right derivation  el anillo de polinomios torcidos que consisten de los polinomios derechos en la variable $t$ sobre $\G$ tal que $at = ta^{\sigma} + a^{\delta}$ para cualquier $a\in\G$, donde $\sigma$ es un automorfismo de $\G$ y $\delta$ es una $\sigma$-derivación sobre $\G$. Se denota por $\mathcal{I}$ un subconjunto maximal de los polinomios indescomponibles no similares dos a dos en $\G [t, \sigma , \delta ]$.

\begin{equation}\label{Series K8(1)}
\renewcommand{\arraystretch}{1.1}\renewcommand{\tabcolsep}{2mm}
\setlength{\doublerulesep}{0.08pt}\begin{tabular}[c]{||c||c||cc||}
  \multicolumn{1}{c}{$\circ$} & \multicolumn{1}{c}{$\circ$} & \multicolumn{2}{c}{$\otimes$} \\ \hline\hline
  $I_n$&&$\xi I_n$&$I_n$ \\
  &$I_n$&$\xi$X&$I_n$ \\ \hline\hline
  \multicolumn{4}{c}{(a)} \\
\end{tabular} \hspace{1cm} \begin{tabular}[c]{||c||c||cc||}
  \multicolumn{1}{c}{$\circ$} & \multicolumn{1}{c}{$\circ$} & \multicolumn{2}{c}{$\otimes$} \\ \hline\hline
  $I_n$&&$\xi I_n$&$I_n$ \\
  &$I_n$&$\overline{\xi}I_n + \xi \overline{X}$&$I_n + \overline{X}$ \\ \hline\hline
    \multicolumn{4}{c}{(b)} \\
\end{tabular}\end{equation} \begin{equation}\label{Series K8(2)} \renewcommand{\arraystretch}{1.1}\renewcommand{\tabcolsep}{2mm}
\setlength{\doublerulesep}{0.08pt} \begin{tabular}[c]{||c||c||cc||}
  \multicolumn{1}{c}{$\circ$} & \multicolumn{1}{c}{$\circ$} & \multicolumn{2}{c}{$\otimes$} \\ \hline\hline
  $I_n$&&$\xi I_n$&$I_n$ \\
  &$I_n$&$\xi [ I_n + \overline{X} ]$&$\overline{X}$ \\ \hline\hline
\end{tabular}
\end{equation}

\end{proof}

\begin{prop}\label{A26A27coreps}
Each of the posets $A_{26}$ and $A_{27}$ and their duals has 6 types of indecomposable corepresentations sincere at the points $\eta$ and $\zeta$ respectively, having the following matrix forms:
\end{prop}

\begin{table}
\renewcommand{\arraystretch}{1.2}
\renewcommand{\tabcolsep}{1.5mm}
\setlength{\doublerulesep}{0.14pt}
\setlength{\arrayrulewidth}{0.2pt}
\begin{tabular}{|rcl|rcl|}\hline
$_{2n}(A_{26}-1)  $ & = & $_{2n}(K_8 - 1^*)^{\eta}    $ & $_{2n}(A_{26}^*-1)  $ & = & $ _{2n}(K_8 - 1)_{\eta}     $\\
$_{2n-1}(A_{26}-2)$ & = & $_{2n-1}(K_8 - 3^*)^{\eta}  $ & $_{2n-1}(A_{26}^*-2)$ & = & $ _{2n-1}(K_8 - 3)_{\eta}   $\\
$_{2n}(A_{26}-3)  $ & = & $_{2n}(K_8 - 5)^{\eta}      $ & $_{2n}(A_{26}^*-3)  $ & = & $ _{2n}(K_8 - 5)_{\eta}     $\\
$_{2n-1}(A_{26}-4)$ & = & $_{2n-1}(K_8 - 8^*)^{\eta}  $ & $_{2n-1}(A_{26}^*-4)$ & = & $ _{2n-1}(K_8 - 8)_{\eta}   $\\
$_{2n-1}(A_{26}-5)$ & = & $_{2n-1}(K_8 - 10)^{\eta}   $ & $_{2n-1}(A_{26}-5)^*$ & = & $ _{2n-1}(K_8 - 10^*)_{\eta}$\\ \hline
$_{2}(A_{26}-6)$    & = & $(F_{3}-B)$                   &                       &   &                              \\
$_{2n+2}(A_{26}-6)$ & = & $\Int \, _{2n+1}(A_{27}-2)  $ & $_{2n+2}(A_{26}^*-6)$ & = & $      {2n+2}(A_{26}-6)^*   $\\ \hline
\end{tabular}
\caption{Classification of indecomposables of $A_26$}\label{List of coreps of A26}
\end{table}

\begin{prop}\label{A40coreps}
Each of the posets $A_{40}$ and $A^*_{40}$ has 15 types of indecomposable corepresentations sincere at the point $p$, recursively defined in the following form
\end{prop}

\begin{table}
\renewcommand{\arraystretch}{1.2}
\renewcommand{\tabcolsep}{1.5mm}
\setlength{\doublerulesep}{0.14pt}
\setlength{\arrayrulewidth}{0.2pt}
\begin{tabular}{|rcl|rcl|}\hline
$_1(A_{40}^*-1) $ & = & $F_{17}     $ & $_{2n+1}(A_{40}^*-1) $ & = & $ \Int \, _{2n}(K_{8} - 1)     $\\
$_2(A_{40}^*-2) $ & = & $F_{15} - D $ & $_{2n+2}(A_{40}^*-2) $ & = & $ \Int \, _{2n+1}(A_{40}^*-1)  $\\
$_2(A_{40}^*-3) $ & = & $F_{15} - A $ & $_{2n+2}(A_{40}^*-3) $ & = & $ \Int \, _{2n+1}(K_{8} - 2)   $\\
$_3(A_{40}^*-4) $ & = & $F_{15} - F $ & $_{2n+3}(A_{40}^*-4) $ & = & $ \Int \, _{2n+1}(K_{8} - 3)   $\\
$_1(A_{40}^*-5) $ & = & $F_{13} - A $ & $_{2n+1}(A_{40}^*-5) $ & = & $ \Int \, _{2n}(K_{8} - 4)     $\\
$_2(A_{40}^*-6) $ &   &               & $_{2n}(A_{40}^*-6)   $ & = & $ \Int \, _{2n-1}(K_{8} - 7)   $\\
$_1(A_{40}^*-7) $ & = & $F_{14} - A $ & $_{2n+1}(A_{40}^*-7) $ & = & $ \Int \, _{2n}(A_{40}^*-6)    $\\
$_1(A_{40}^*-8) $ & = & $F_{13} - B $ & $_{2n+1}(A_{40}^*-8) $ & = & $ \Int \, _{2n-1}(K_{8} - 8)   $\\
$_3(A_{40}^*-9) $ & = & $F_{15} - G $ & $_{2n+3}(A_{40}^*-9) $ & = & $ \Int \, _{2n+1}(A_{40}^*-8)  $\\
$_2(A_{40}^*-10)$ & = & $F_{14} - C $ & $_{2n+2}(A_{40}^*-10)$ & = & $ \Int \, _{2n}(K_{8} - 9)     $\\
$_2(A_{40}^*-11)$ &   &               & $_{2n}(A_{40}^*-11)  $ & = & $ \Int \, _{2n-1}(A_{40}^*-13) $\\
$_2(A_{40}^*-12)$ & = & $F_{15} - C $ & $_{2n+2}(A_{40}^*-12)$ & = & $ \Int \, _{2n}(A_{40}^*-11)   $\\
$_1(A_{40}^*-13)$ & = & $F_{15} - B $ & $_{2n+1}(A_{40}^*-13)$ & = & $ \Int \, _{2n}(A_{40}^*-12)   $\\
$_2(A_{40}^*-14)$ & = & $F_{15} - E $ & $_{2n+2}(A_{40}^*-14)$ & = & $ \Int \, _{2n}(A_{40}^*-14)   $\\
$_1(A_{40}^*-15)$ & = & $F_{14} - B $ & $_{2n+1}(A_{40}^*-15)$ & = & $ \Int \, _{2n-1}(A_{40}^*-15) $\\ \hline
\end{tabular}
\caption{Classification of indecomposables of $A_{40}^*$, $n\ge 1$.}\label{List of coreps of A40 dual}
\end{table}

\begin{equation}
\renewcommand{\arraystretch}{1.2}
\renewcommand{\tabcolsep}{1.5mm}
\setlength{\doublerulesep}{0.14pt}
\setlength{\arrayrulewidth}{0.2pt}
\begin{tabular}{|rcl|rcl|}\hline
$ _{2n}(A_{40}-2)  $ & = &  $_{2n}(K_{8}-1^*)^q $               & $_{2n}(A_{40}-3)  $ & = & $_{2n}(K_{8}-4)^q     $ \\
$ _{2n+1}(A_{40}-4)$ & = &  $_{2n+1}(K_{8}-\widetilde{3}^*)^q$  & $_{2n-1}(A_{40}-5)$ & = & $_{2n-1}(K_{8}-2)^q   $ \\
$ _{2n-1}(A_{40}-7)$ & = &  $_{2n-1}(K_{8}-7^*)^q$              & $_{2n-1}(A_{40}-8)$ & = & $_{2n-1}(K_{8}-8^*)^q $ \\
$ _{2n}(A_{40}-10) $ & = &  $_{2n}(K_{8}-9^*)^q$                &                     &   &                         \\ \hline
\end{tabular}
\end{equation}

\begin{lemma}
If $\p$ is an equipped poset which coincides with $A_{40}$ (as in table \ref{listPosets}) then each indecomposable corepresentation (probably not sincere) $U$ of $\p$ satisfies the following conditions
\begin{enumerate}
\item[(a)] $U_a = 0$ or $\US{U}_a + U_{\sigma} = U_0$
\item[(b)] $U_{\rho} = 0$ ,or , $U_q + U_{\sigma} = U_q + \US{U}_{a} = U_0$ and $\US{U}_q = U_0$
\end{enumerate}
\end{lemma}


\begin{equation}
\renewcommand{\arraystretch}{1.2}
\renewcommand{\tabcolsep}{1.5mm}
\setlength{\doublerulesep}{0.14pt}
\setlength{\arrayrulewidth}{0.2pt}
\begin{tabular}{|rcl|rcl|}\hline
$_{2n-1}(A_{27}-1)$ & = & $_{2n-1}(K_8 - 7^*)^{\zeta}   $ & $_{2n-1}(A_{27}^*-1)$ & = & $ _{2n-1}(K_8 - 7)_{\zeta}  $\\
$_{2n-1}(A_{27}-2)$ & = & $_{2n-1}(K_8 - 8^*)^{\zeta}   $ & $_{2n-1}(A_{27}^*-2)$ & = & $ _{2n-1}(K_8 - 8)_{\zeta}  $\\
$_{2n}(A_{27}-3)  $ & = & $_{2n}(K_8 - 9^*)^{\zeta}     $ & $_{2n}(A_{27}^*-3)  $ & = & $ _{2n}(K_8 - 9)_{\zeta}    $\\
$_{2n-1}(A_{27}-4)$ & = & $_{2n-1}(K_8 - 10^*)^{\zeta}  $ & $_{2n-1}(A_{27}^*-4)$ & = & $ _{2n-1}(K_8 - 10)_{\zeta} $\\
$_{2n}(A_{27}-5)  $ & = & $_{2n}(K_8 - 11)^{\zeta}      $ & $_{2n}(A_{27}^*-5)  $ & = & $ _{2n}(K_8 - 11)_{\zeta}   $\\ \hline
$_{2}(A_{27}-6)$    & = & $(F_{14}-C)$                    &                       &   &                              \\
$_{2n+2}(A_{27}-6)$ & = & $\Int \, _{2n+1}(A_{26}-4)    $ & $_{2n+2}(A_{27}^*-6)$ & = & $      {2n+2}(A_{27}-6)^*   $\\ \hline
\end{tabular}
\end{equation}


\begin{prop}\label{A30'scoreps}
Each of the posets $A_{30}$, $A_{31}$, $A_{32}$, $A_{35}$, $A_{36}$, $A_{37}$, $A_{47}$ and their duals has 1 type of indecomposable corepresentations sincere at all its strong points having the following forms:
\end{prop}

\begin{equation}
\renewcommand{\arraystretch}{1.2}
\renewcommand{\tabcolsep}{1.5mm}
\setlength{\doublerulesep}{0.14pt}
\setlength{\arrayrulewidth}{0.2pt}
\begin{tabular}{|rcl|rcl|rcl|}\hline
$_{2n}(A_{30})$ & = & $_{2n}(A_{26} - 6)^{\theta}$ & $_{2n-1}(A_{31})$ & = & $ _{2n-1}(A_{26} - 4)^{\zeta} $ & $_{2n}(A_{32})$ & = & $ _{2n}(A_{27} - 6)^{\theta}$ \\
$_{2n}(A_{30}^*)$ & = & $_{2n}(A_{26}^* - 6)_{\theta}$ & $_{2n-1}(A_{31}^*)$ & = & $_{2n-1}(A_{26}^* - 4)_{\zeta}$ & $_{2n}(A_{32}^*)$ & = & $ _{2n}(A_{27}^* - 6)_{\theta}$ \\
$_{2n-1}(A_{35})$ & = & $ _{2n-1}(A_{26} - 5)_{\zeta}$ & $_{2n}(A_{36})$ & = & $_{2n}(A_{26}^* - 3)^{\theta}$ & $_{2n}(A_{37})$ & = & $ _{2n}(A_{27}^* - 5)^{\theta}$ \\
$_{2n-1}(A_{47})$ & = & $_{2n-1}(A_{26}^* - 5)^{\zeta}$ & $_{2n}(A_{36}^*)$ & = & $_{2n}(A_{26} - 3)_{\theta}$ & $_{2n}(A_{37}^*)$ & = & $ _{2n}(A_{27} - 5)_{\theta}$ \\ \hline
\end{tabular}
\end{equation}

\begin{lemma}
Let $\p$ be an equipped poset which coincides with one of the sets $K_8$, $A_{26}$, $A_{27}$, $A_{30}$, $A_{31}$ or $A_{32}$ (denoted as in table \ref{listPosets}). Then, each indecomposable corepresentation (probably not sincere) $U$ of $\p$ satisfies the following conditions
\begin{enumerate}
\item[(a)] $U_a = 0$ or $\US{U}_a + U_{\rho} = \US{U}_a + U_{\sigma} = U_0$
\item[(b)] $U_r = 0$ or $\begin{cases} U_a + U_{\rho} + U_{\sigma} = U_0 \,\, \text{for any}\, \p \\ \US{U}_a + U_{\rho} = U_0\,\, \text{for}\, \p = K_8, A_{27}, A_{32} \end{cases}$
\end{enumerate}
\end{lemma}

\begin{prop}\label{A43'scoreps}
Each of the posets $A_{43}$, $A_{44}$ and $A_{48}$ has 1 type of indecomposable corepresentations sincere at all its weak points having the following forms:
\begin{align*}
_{2}(A_{43}) &=\, F_{18} & _{2n+2}(A_{43}) &=\, \Int \, _{2n+1}(A_{40} - 5) \\
_{1}(A_{44}) &=\, F_{17} & _{2n+1}(A_{44}) &=\, \Int \, _{2n}(A_{40} - 3)   \\
             &           & _{2n+1}(A_{48}) &=\, \Int \, _{2n}(A_{40} - 3)
\end{align*} for $n\ge 1$.
\end{prop}

\section{Classification of indecomposable corepresentations of $K_9$}\label{sectionK9}

\begin{figure}[h]
\centering
\begin{picture}(250,100)

\multiput(0.35,2)(0,30){2}{$\otimes$}
\put(5,8.8){\line(0,1){22.3}}
\multiput(32,2)(0,30){2}{$\circ$}
\put(35,7){\line(0,1){25.9}}

\put(62,40){D-$\seven$}
\put(60,35){\vector(1,0){40}}
\put(65,25){$(a,\sigma)$}

\put(112,32){$\circ$}
\multiput(140.35,2)(30,30){2}{$\otimes$}
\put(140.35,62){$\otimes$}
\multiput(202,62)(30,-30){2}{$\circ$}
\put(147.8,7.8){\line(1,1){24.4}}
\put(147.8,62.2){\line(1,-1){24.4}}
\put(116.5,36.5){\line(1,1){25.7}}
\put(116.5,33.5){\line(1,-1){25.7}}
\put(206.5,63.5){\line(1,-1){26.9}}
\put(177.8,37.8){\line(1,1){25.7}}

\put(10,33){{\footnotesize $p$}}
\put(10,3){{\footnotesize  $a$}}
\put(40,33){{\footnotesize  $\sigma$}}
\put(40,3){{\footnotesize  $\rho$}}

\put(150,63){{\footnotesize $p^+$}}
\put(120,33){{\footnotesize  $a^+$}}
\put(240,33){{\footnotesize  $\rho$}}
\put(210,63){{\footnotesize $\sigma$}}
\put(180,33){{\footnotesize $p^-$}}
\put(150,3){{\footnotesize  $a^-$}}

\end{picture}
\caption{Differentiation of $K_9$}\label{K9 derivado}
\end{figure}
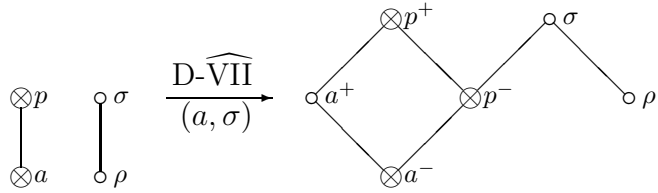

\begin{theorem}
The critical poset $K_9$ possesses 48 types of indecomposable corepresentations over the pair ($\F ,\ \G$), listed together with their dual corepresentations in the inductive definition
\end{theorem}

\begin{table}[hbtp]
\centering
\renewcommand{\arraystretch}{1.2}
\renewcommand{\tabcolsep}{1.5mm}
\setlength{\doublerulesep}{0.14pt}
\setlength{\arrayrulewidth}{0.2pt}
\begin{tabular}{|rcl|rcl|}\hline
$(K_9 - 1) $ & = & $\Int\, (K_8 - 1)     $  & $(K_9 - 1^*) $ & = & $\Int\, (A_{40} - 6)     $ \\
$(K_9 - 2) $ & = & $\Int\, (K_8 - 1^*)   $  & $(K_9 - 2^*) $ & = & $\Int\, (A_{40^*} - 6)   $ \\
$(K_9 - 3) $ & = & $\Int\, (K_8 - 2)     $  & $(K_9 - 3^*) $ & = & $\Int\, (A_{40^*} - 3)   $ \\
$(K_9 - 4) $ & = & $\Int\, (K_8 - 2^*)   $  & $(K_9 - 4^*) $ & = & $\Int\, (A_{40} - 3)     $ \\
$(K_9 - 5) $ & = & $\Int\, (K_8 - 7)     $  & $(K_9 - 5^*) $ & = & $\Int\, (A_{40} - 1)     $ \\
$(K_9 - 6) $ & = & $\Int\, (K_8 - 7^*)   $  & $(K_9 - 6^*) $ & = & $\Int\, (A_{40^*} - 1)   $ \\
$(K_9 - 7) $ & = & $\Int\, (A_{40} - 2)  $  & $(K_9 - 7^*) $ & = & $\Int\, (A_{40^*} - 7)   $ \\
$(K_9 - 8) $ & = & $\Int\, (A_{40} - 7)  $  & $(K_9 - 8^*) $ & = & $\Int\, (A_{40^*} - 2)   $ \\
$(K_9 - 9) $ & = & $\Int\, (A_{40} - 11) $  & $(K_9 - 9^*) $ & = & $\Int\, (A_{40^*} - 13)  $ \\
$(K_9 - 10)$ & = & $\Int\, (A_{40} - 12) $  & $(K_9 - 10^*)$ & = & $\Int\, (A_{40^*} - 11)  $ \\
$(K_9 - 11)$ & = & $\Int\, (A_{40} - 13) $  & $(K_9 - 11^*)$ & = & $\Int\, (A_{40^*} - 12)  $ \\
$(K_9 - 12)$ & = & $\Int\, (A_{40} - 5)  $  &                &   &                            \\
$(K_9 - 13)$ & = & $\Int\, (A_{40^*} - 5)$  &                &   &                            \\ \hline
\end{tabular}
\caption{Corepresentaciones del poset $K_9$ y forma de Tits con valor 1, los tipos de corepresentación 12 y 13 son autoduales.}\label{coreps de K9 f1}
\end{table}

\begin{table}[hbtp]
\centering
\renewcommand{\arraystretch}{1.2}
\renewcommand{\tabcolsep}{1.5mm}
\setlength{\doublerulesep}{0.14pt}
\setlength{\arrayrulewidth}{0.2pt}
\begin{tabular}{|rcl|rcl|}\hline
$(K_9 - 14) $ & = & $\Int\, (K_8 - 3)               $  & $(K_9 - 14^*) $ & = & $\Int\, (K_8 - \widetilde{9}^*)  $ \\
$(K_9 - 15) $ & = & $\Int\, (K_8 - 3^*)             $  & $(K_9 - 15^*) $ & = & $\Int\, (K_8 - \widetilde{9})    $ \\
$(K_9 - 16) $ & = & $\Int\, (K_8 - \widetilde{3})   $  & $(K_9 - 16^*) $ & = & $\Int\, (A_{40} - 10)            $ \\
$(K_9 - 17) $ & = & $\Int\, (K_8 - \widetilde{3}^*) $  & $(K_9 - 17^*) $ & = & $\Int\, (A_{40^*} - 10)          $ \\
$(K_9 - 18) $ & = & $\Int\, (K_8 - 8)               $  & $(K_9 - 18^*) $ & = & $\Int\, (A_{40} - 8)             $ \\
$(K_9 - 19) $ & = & $\Int\, (K_8 - 8^*)             $  & $(K_9 - 19^*) $ & = & $\Int\, (A_{40^*} - 8)           $ \\
$(K_9 - 20) $ & = & $\Int\, (K_8 - 9)               $  & $(K_9 - 20^*) $ & = & $\Int\, (A_{40} - 4)             $ \\
$(K_9 - 21) $ & = & $\Int\, (K_8 - 9^*)             $  & $(K_9 - 21^*) $ & = & $\Int\, (A_{40^*} - 4)           $ \\
$(K_9 - 22) $ & = & $\Int\, (A_{40} - 9)            $  & $(K_9 - 22^*) $ & = & $\Int\, (A_{40^*} - 9)           $ \\
$(K_9 - 23)$ & = & $\Int\,  (A_{40} - 14)           $  & $(K_9 - 23^*)$ & = & $\Int\, (A_{40^*} - 15)           $ \\
$(K_9 - 24)$ & = & $\Int\,  (A_{40} - 15)           $  & $(K_9 - 24^*)$ & = & $\Int\, (A_{40^*} - 14)           $ \\
$(K_9 - 25)$ & = & $\Int\,  (K_8 - 10)              $  &                &   &                                     \\
$(K_9 - 26)$ & = & $\Int\,  (K_8 - 10^*)            $  &                &   &                                     \\ \hline
\end{tabular}
\caption{Corepresentaciones del poset $K_9$ y forma de Tits con valor 2, los tipos de corepresentación 25 y 26 son autoduales.}\label{coreps de K9 f2}
\end{table}

\begin{center}
\begin{table}[hbtp]
\renewcommand{\arraystretch}{1.1}
\renewcommand{\tabcolsep}{2mm}
\begin{tabular}{llll}
$W_{1}\cong \ _1(K_9 - 6)$,    & $r_{W_1} = (0,0,0,1)$; &  $W_{6}\cong \ _1(K_9 - 16^*)$, & $r_{W_6} = (0,1,0,0)$; \\
$W_{2}\cong \ _1(K_9 - 7)$,    & $r_{W_2} = (0,0,1,0)$; &  $W_{7}\cong \ _1(K_9 - 23)$,   & $r_{W_7} = (1,1,2,2)$; \\
$W_{3}\cong \ _1(K_9 - 9)$,    & $r_{W_3} = (0,1,1,1)$; &  $W_{8}\cong \ _1(K_9 - 15)$,   & $r_{W_8} = (1,0,0,0)$; \\
$W_{4}\cong \ _2(K_9 - 11)$,   & $r_{W_4} = (0,0,0,1)$; &  $W_{9}\cong \ _1(K_9 - 19)$,   & $r_{W_9} = (1,0,0,2)$; \\
$W_{5}\cong \ _2(K_9 - 18^*)$, & $r_{W_5} = (0,1,0,0)$; &  $W_{10}\cong \ _2(K_9 - 24)$,  & $r_{W_{10}} = (0,1,0,2)$; \\
\end{tabular}
\caption{Corepresentaciones de $K_9$ con $r\not=0$.}\label{K9 coreps r no 0}
\end{table}
\end{center}

\begin{lemma}
An indecomposable corepresentation $U$ of a poset $\p$ of the form (\ref{K9 subset}) is not sincere at $K_9^+$ or $K_9^-$
\end{lemma}

\begin{lemma}\label{K9 the one}
If $\p$ is a sincere one parametric equipped poset, containing a subset $K_9$, then $\p = K_9$.
\end{lemma}

\section{Proof of the Main Theorems}\label{proofs}

The Theorems A and B are followed from results about classification of indecomposables corepresentations of 2-equipped posets given in Sections \ref{sectionK6} - \ref{sectionK9}, listed in appendixes \ref{sincereK6posets}, \ref{sincereK8posets} and \ref{specialposets}, as well as the Lemmas \ref{K7 the one}, \ref{K9 the one} above, and Theorems \ref{K6 subset}, \ref{K8 subset} given in this Section. Firstly, we proof some lemmas.


\begin{lemma}
Let $\p$ be an equipped poset of the form $ \p = \{ a; \: b\} + \{ a;\: b\}^{\curlyvee} + \{a; \: b\}_{\curlywedge} $ where $a,\ b$ are weak points, and $\{a; \: b\}_{\curlywedge} = p_m \prec \cdots \prec p_1$ and $\{a; \: b\}^{\curlyvee} = q_1 \prec \cdots \prec q_n$ are completely weak chains. Then, $\p$ is sincere if and only if $n+m \le 2$.\label{K6 weak well ins. sets}
\end{lemma}

\begin{proof} We study two cases, \textbf{(1)} $\{a; \: b\}_{\curlywedge} = \emptyset$ or $\{a; \: b\}^{\curlyvee} = \emptyset$. We suppose $\{a; \: b\}_{\curlywedge} = \emptyset$ (the second condition is dual). We apply induction on $n\ge 2$. Notice that $A_{38}$ and $A_{42}$ correspond to the base case. As the induction hypothesis we assume that $\p$ is not sincere for $2 < n < k$, for some fixed number $k$. Let $U$ be an indecomposable corepresentation of $\p$ and decompose its restriction $V = U|_{\p \setminus q_n}$ in a direct sum of indecomposables $ V = \bigoplus_i V^i$. By the induction hypothesis, it follows that the support of each summand is sincere at most in one point of $q_1, \ldots , q_{n-1}$, otherwise it would be a direct summand of $U$. We consider the matrix realization of $U$ by writing in the canonical form the summands $V^i$. In the vertical stripe corresponding to $q_n$, it is obtained a matrix problem (analogously to the describe above in the case of the poset $A_{42}$, see the Figure \ref{guirnalda}) about classification of indecomposable representations of a garland $\mathcal{G}'$, which differs from the garland (\ref{guirnalda orden}) by substituting the chain $b_{i_1} < \cdots < b_{i_{k-1}}$ for each point $b_i$. Its classification has the form (\ref{classification of reps of a garland}) together with identity matrix blocks of order 1 on the horizontal stripes corresponding to the points $b_{n_j}$. It implies that $U$ is decomposable.

Case \textbf{(2)} $|\{a; \: b\}_{\curlywedge}| > 0$ and $|\{a; \: b\}^{\curlyvee}| > 0$. Consider a matrix realization of an indecomposable corepresentation $U$ of $\p$ sincere at $q_n$, and apply induction on $m+n \ge 2$. The base case corresponds to the posets $A_{41}$ and $A_{46}$. Given a fixed $k > 2$, we assume that $\p$ is not sincere provided that $2 < m + n < k$. For the induction step, we decompose the restriction $V = U|_{\p\setminus q_n}$ in indecomposable summands $V = \bigoplus_i V^i$. Since $U$ is indecomposable, the summands have the form $(K_6 - 1)$, $(K_6 - 1^*)$, $(K_6 - 2^*)$, $(K_6 - 5)$ or $(A_{38} - 1)$ with support at the sets $\{ a; \: b \} + q_i$ with $i =1, \ldots , n-1$, or $(A_{38}^* - 1)$ with support at the sets $\{ a; \: b \} + q_j$ with $j =1, \ldots , m$. On the vertical stripe corresponding to the point $q_n$ appears a problem about the classification of indecomposable representations of a garland $\mathcal{G}''$ which is obtained from the garland $\mathcal{G}'$ by changing the chain $d_n < \cdots < d_1$ for the chain $d_n < f_n < f_n' < \cdots < d_2 < f_2 < f_2' < d_1 < f_n < f_n' $. The horizontal stripes $f_i, f_i'$ correspond to the corepresentation summands $(A_{38}^* - 1)$. The problem determined by the garland $\mathcal{G}''$ is reduced in the same way that garlands $\mathcal{G}'$ and (\ref{guirnalda orden}).
\end{proof}

If an equipped poset $\p$ contains a critical subset $\{ a, b \}$, we denote $A^- = N(b)\setminus a^{\vee}$, $A^+ = N(b)\setminus a_{\wedge} $, where $N(b)$ is the set of incomparable points of $b$ in $\p$, $B^-$ and $B^+$ are defined analogously.

\begin{lemma}
Let $\p$ be an equipped poset containing a critical subset $K_{6} = \{ a, b \}$ and having the form \[ \p = \{ A^- \lhd a \lhd A^+ , \,\,\,  B^- \lhd b \lhd B^+ \}, \] where $A^-$, $A^+$, $B^-$, $B^+$ are chains of strong points. Then, $\p$ is sincere if and only if $|A^-| + |A^+| + |B^-| + |B^+| \le 2$.\label{K6 strong well ins. sets}
\end{lemma}

We recall from \cite{Rod-10}, that a subset $K_6 \subset \p $ is said to be \textbf{well inserted} in $\p$ if the subsets \begin{align*}
(a\curlyvee b) &= \{  x\in\p\, : \, a \precneqq x , \:\: \: b\precneqq x  \}, & (a\curlywedge b) &= \{  x\in\p\, : \, x\precneqq a ,\: \: x\precneqq b  \}
\end{align*} are chains, and each of the subsets $N(a)$ and $N(b)$ is a chain with a unique weak point. The corollary below follows easily from the lemmas \ref{K6 weak well ins. sets}, \ref{K6 strong well ins. sets} and the fact that each indecomposable corepresentation of  an equipped poset $\p$ containing a well inserted subset $ K_6$ and two points $p,q$ such that $p \vartriangleleft a$ and  $\{ a, b\} < q$, or,  $p < \{a, b \}$ and $a \vartriangleleft q$ is not sincere at one of the points $p,q$, see corollary 3.2 in \cite{Rod-10}.

\begin{corol}\label{One-par = Mi type}
Let $\p $ be an one-parametric equipped poset of the form \[ \p = K_6 + (a\curlyvee b) + (a \curlywedge b) + A^+ + A^- + B^+ + B^-. \] Then, $\p$ is sincere if and only if $|(a\curlyvee b)| + |(a \curlywedge b)| + | A^+ | + |A^- | + | B^+ | + |B^- | \le 2$
\end{corol}

\begin{remark}
If $\p$ is in the form of corollary \ref{One-par = Mi type}, the above equivalence can be rewrite as follows: $\p$ is sincere if and only if $\p$ is isomorphic or anti-isomorphic to $K_6$, $A_{25}$, $A_{28}$, $A_{29}$, $A_{33}$, $A_{34}$, $A_{38}$, $A_{39}$, $A_{41}$, $A_{42}$, $A_{45}$ or $A_{46}$
\end{remark}

\begin{theorem}\label{K6 subset}
An equipped poset $\p$ of one parameter containing a critical subset $K_6$ is sincere if and only if it is isomorphic or antiisomorphic to some of the sets $A_{25}$, $A_{28}$, $A_{29}$, $A_{33}$, $A_{34}$, $A_{38}$, $A_{39}$, $A_{41}$, $A_{42}$, $A_{45}$ and $A_{46}$.
\end{theorem}

\begin{lemma}
An equipped poset $\p$ of one parameter of the form \[ \p = K_8 + \Omega^+ + \Omega^- + \Gamma^- + \Gamma^+ + \Delta^- + \Delta^+ \] is sincere if and only if it is isomorphic or antiisomorphic to one of the sets $K_8$, $A_{26}$, $A_{27}$, $A_{30}$, $A_{31}$, $A_{32}$, $A_{35}$, $A_{36}$, $A_{37}$, $A_{40}$, $A_{43}$, $A_{44}$, $A_{47}$ or $A_{48}$.
\end{lemma}

\begin{theorem}\label{K8 subset}
Let $\p$ be a one parametric equipped poset containing a critical subset $K_8$. $\p$ is sincere if and only if it is isomorphic or antiisomorphic to some of the sets $A_{26}$, $A_{27}$, $A_{30}$, $A_{31}$, $A_{32}$, $A_{35}$, $A_{36}$, $A_{37}$, $A_{40}$, $A_{43}$, $A_{44}$, $A_{47}$ and $A_{48}$.
\end{theorem}

\appendix

\section{The list of finite-type 2-equipped posets and their classification of indecomposable corepresentations}\label{finitetypeposets}

Indecomposable corepresentations of a one-parameter 2-equipped poset $\p$ (except series) are obtained in a recursive way, starting from indecomposable corepresentations of $\p$ whose support is a finite-type poset $\q \subset \p$ via integration algorithm (opposite to the differentiation $\seven$, obtained in \cite{Rod-Zav-07}), roughly speaking most of discrete classification of one-parameter 2-equipped posets comes from classification of corepresentations of finite-type 2-equipped posets.

In order that results presented in this paper will be clear, a list of finite-type 2-equipped posets together with their classification of indecomposable corepresentations are presented. The list can be obtained by following results in \cite{Kle-Sim-90} about subspace categories of a vector space categories, see \cite{Naz-Roi-72}, \cite{Ri84}, \cite{Sim92}. While a complete description of their indecomposable matrix corepresentations was obtained in \cite{Rod-Zav-07} in a closed terminology used here.\\

\noindent

\textbf{The list of finite-type 2-equipped posets}

\noindent
\scalebox{.85}{
\begin{picture}(495,310)
\put(2,183){\small{$f$=$1$}}%
\put(2,59){\small{$f$=$2$}}%
\put(495,0){\line(0,1){310}}%
\put(415,0){\line(0,1){310}}%
\put(370,0){\line(0,1){310}}%
\put(295,0){\line(0,1){310}}%
\put(135,0){\line(0,1){310}}%
\put(70,0){\line(0,1){310}}%
\put(25,0){\line(0,1){310}}%
\put(0,0){\line(0,1){310}}%
\put(0,310){\line(1,0){495}}%
\put(0,240){\line(1,0){495}}%
\put(0,130){\line(1,0){495}}%
\put(0,0){\line(1,0){495}}%

\put(28,298){{\small $F_{13} $}}%
\put(72,298){{\small $F_{14} $}}%
\put(138,298){{\small $F_{15} $}}%
\put(298,298){{\small $F_{16} $}}%
\put(373,298) {{\small $F_{17} $}}%
\put(418,298) {{\small $F_{18}$}}%

\put(40.35,262){$\otimes $}
\put(50,263){$a$}%

\put(79.35,262){$\otimes$}
\put(111,262){$\circ$}%
\put(89,263){$a$}%
\put(119,263){$\zeta$}%

\multiput(195.35,262)(0,20){2}{$\otimes$}
\put(200,268.8){\line(0,1){12.3}}%
\put(227,262){$\circ $}%
\put(205,263){$a$}%
\put(205,283){$ b$}%
\put(235,263){$\zeta$}%

\put(310.35, 262){$\otimes$}
\multiput(342,262)(0,20){2}{$\circ $}%
\put(345,267.3){\line(0,1){15.8}}%
\put(320,263){$a$}%
\put(350,263){$\zeta$}%
\put(350,283){$\eta$}%

\multiput(385.35,262)(0,20){2}{$\otimes$}
\put(395, 263){$ a$}%
\put(395, 283){$ b$}%
\put(390,268.8){\line(0,1){12.3}}%

\multiput(440.35,252)(0,20){3}{$\otimes$}
\multiput(445,258.8)(0,20){2}{\line(0,1){12.3}}%
\qbezier(444,259)(424,275)(444,291)%
\put(472,252){$\circ$}%
\put(450,253){$ a$}%
\put(450,273){$ c$}%
\put(450,293){$ b$}%
\put(480,253){$\zeta$}%

%
%
\put(35,210){%
\small{%
\begin{tabular}{|c|}\h%
1 \\ \h%
\end{tabular}}}%
\put(38,195){{\small $A$}}%
\put(30,100){%
\small{%
\begin{tabular}{|cc|}\h%
1 & $\xi$ \\ \h%
\end{tabular}}}%
\put(38,85){{\small $B$}}%
\put(85,210){%
\small{%
\begin{tabular}{|c|c|}\h%
1 & 1\\ \h%
\end{tabular}}}%
\put(98,195){{\small $A$}}%
\put(75,100){%
\small{%
\begin{tabular}{|cc|c|}\h%
1 & $\xi$ & 1\\ \h%
\end{tabular}}}
\put(98,85){{\small $B$}}%
\put(75,40){%
\small{%
\begin{tabular}{|cc|c|}\h%
1 & 0 & 1\\%
1 & $\xi$ & 0 \\ \h%
\end{tabular}}}%
\put(98,18){{\small $C$}}%

\put(149,215){%
\small{%
\begin{tabular}{|c|c|c|}\h%
1 & 0 & $\xi$ \\%
0 & 1 & 1 \\ \h%
\end{tabular}}}%
\put(169,193){{\small $A$}}%
\put(140,160){%
\small{%
\begin{tabular}{|cc|c|c|}\h%
1 & 0 & $\xi$ & $\xi$ \\%
0 & 1 & 0 & 1 \\ \h%
\end{tabular}}}%
\put(169,138){{\small $C$}}%
\put(226,215){%
\small{%
\begin{tabular}{|c|c|c|}\h%
1  & $\xi$ & 1 \\ \h%
\end{tabular}}}%
\put(247,193){{\small $B$}}%
\put(220,160){%
\small{%
\begin{tabular}{|c|cc|c|}\h%
0 & 1 & $\xi$ & 1 \\%
1 & 0 & 0 & 1 \\ \h%
\end{tabular}}}%
\put(250,138){{\small $D$}}%

\put(171,105){%
\small{%
\begin{tabular}{|cc|cc|c|}\h%
1 & 0 & $\xi$ & 0 & $\xi$ \\%
0 & 1 & 0 & $\xi$ & 1 \\ \h%
\end{tabular}}}%
\put(270,105){{\small $E$}}%
\put(171,67){%
\small{%
\begin{tabular}{|cc|cc|c|}\h%
0 & 0 & 1 & $\xi$ & 1 \\%
1 & 0 & 0 & 0 & $\xi$ \\%
0 & 1 & 0 & 0 & 1 \\ \h%
\end{tabular}}}%
\put(270,65){{\small $F$}}%
\put(163,22){%
\small{%
\begin{tabular}{|cc|cc|cc|}\h%
0 & 0 & 1 & $\xi$ & 1 & 0 \\%
1 & 0 & 0 & 0 & 0 & $\xi$ \\%
0 & 1 & 0 & 0 & 1 & 1 \\ \h%
\end{tabular}}}%
\put(275,20){{\small $G$}}%

\put(298,67){%
\small{%
\begin{tabular}{|c|c|c|}\h%
1 \ 0 & $\xi$ & 1 \\%
0 \ 1 & 1 & 0 \\ \h%
\end{tabular}}}%

\put(375,210){%
\small{%
\begin{tabular}{|c|c|}\h%
1 & $\xi$ \\ \h%
\end{tabular}}}%

\put(420,210){%
\small{%
\begin{tabular}{|c|c|c|c|}\h%
1 & 0 & 0 & 1 \\%
0 & 1 & $\xi$ & $\xi$ \\ \h%
\end{tabular}}}%
\end{picture}}

We recall that $\F \subset \F(\xi) = \G$ is an arbitrary quadratic field extension, with generator $\xi $. The classification is separated in two parts according to the value of the Tits form of the vector dimension of the indecomposable corepresentation, $f(d) = 1$ or $2$.

One can refer to a specific matrix corepresentation $M$ of a poset $\p$ listed above by mentioning its respective letter in the form ($\p$-\textit{letter}), provided that it has more than one indecomposable corepresentation, as is the case of $F_{13}$, $F_{14}$ and $F_{15}$, or just writting ($\p$) when is not confusion about its unique indecomposable corepresentation, as is the case of $F_{16}$, $F_{17}$ and $F_{18}$.

The following lemma may be checked inspecting the classification given above.

\begin{lemmaA1}
An indecomposable corepresentation of a finite 2-equipped poset $\p$ which coincides with one of the sets $F_{15}$, $F_{17}$ or $F_{18}$ satisfies the following conditions
\[ \begin{cases} \US{U}_a = U_0\, \text{ and }\, U_b = U_0,  & \ \ \text{if }\, \p = F_{17} \\ \US{U}_b = \US{U}_a + U_{\xi} = U_0, & \ \ \text{if }\, \p = F_{15} \text{ or } F_{18}  \\  \US{U}_a + U_b = U_0, & \ \ \text{if }\, \p = F_{18} \end{cases} \]
\end{lemmaA1}

\section{Sincere one parametric 2-equipped posets containing $K_6$, but critical and special sets}\label{sincereK6posets}

\begin{picture}(430,160)


\put(0,0){\line(1,0){430}}
\put(0,80){\line(1,0){430}}
\put(0,160){\line(1,0){430}}

\put(0,0){\line(0,1){160}}
\put(90,0){\line(0,1){80}}
\put(100,80){\line(0,1){80}}
\put(170,0){\line(0,1){80}}
\put(250,0){\line(0,1){80}}
\put(210,80){\line(0,1){80}}
\put(340,0){\line(0,1){80}}
\put(320,80){\line(0,1){80}}
\put(430,0){\line(0,1){160}}
\put(1,150){$A_{25}$}
\multiput(25.35,102)(30,0){2}{$\otimes$}
\put(57,122){$\circ$}
\put(60,108.8){\line(0,1){14.3}}

\put(101,150){$A_{28}$}
\multiput(130.35,102)(30,0){2}{$\otimes$}
\multiput(132,122)(30,0){2}{$\circ$}
\multiput(135,108.8)(30,0){2}{\line(0,1){14.3}}
\put(140,123){{\footnotesize $1$}}
\put(140,103){{\footnotesize $2k$}}
\put(170,123){{\footnotesize $1$}}
\put(170,103){{\footnotesize $2k$}}
\put(101,83){{\footnotesize $d_0 = 2k+1$}}
\put(180,83){{\footnotesize $f = 2$}}

\put(211,150){$A_{29}$}
\multiput(230.35,102)(30,0){2}{$\otimes$}
\multiput(262,122)(0,20){2}{$\circ$}
\put(265,126.9){\line(0,1){16.2}}
\put(265,108.8){\line(0,1){14.3}}
\put(240,103){{\footnotesize $2k$}}
\put(270,143){{\footnotesize $1$}}
\put(270,123){{\footnotesize $1$}}
\put(270,103){{\footnotesize $2k-2$}}
\put(211,83){{\footnotesize $d_0 = 2k$}}
\put(290,83){{\footnotesize $f = 2$}}

\put(321,150){$A_{33}$}
\multiput(350.35,112)(30,0){2}{$\otimes$}
\multiput(352,132)(30,-40){2}{$\circ$}
\put(355,118.8){\line(0,1){14.3}}
\put(385,111.2){\line(0,-1){14.3}}
\put(360,133){{\footnotesize $1$}}
\put(360,113){{\footnotesize $2k$}}
\put(390,113){{\footnotesize $2k$}}
\put(390,93){{\footnotesize $1$}}
\put(321,83){{\footnotesize $d_0 = k+1$}}
\put(400,83){{\footnotesize $f = 2$}}

\put(1,70){$A_{34}$}
\multiput(15.35,42)(30,0){2}{$\otimes$}
\multiput(47,62)(0,-40){2}{$\circ$}
\put(50,48.8){\line(0,1){14.3}}
\put(50,41.2){\line(0,-1){14.3}}
\put(25,43){{\footnotesize $2k$}}
\put(55,63){{\footnotesize $1$}}
\put(55,43){{\footnotesize $2k-2$}}
\put(55,23){{\footnotesize $1$}}
\put(1,3){{\footnotesize $d_0 = 2k$}}
\put(60,3){{\footnotesize $f = 2$}}

\put(91,70){$A_{38}$}
\multiput(110.35,22)(30,0){2}{$\otimes$}
\put(125.35,42){$\otimes$}
\put(117.4,28.2){\line(3,4){10.2}}
\put(142.6,28.2){\line(-3,4){10.2}}

\put(171,70){$A_{39}$}
\multiput(180.35,22)(30,0){2}{$\otimes$}
\put(195.35,42){$\otimes$}
\put(227,42){$\circ$}
\put(187.4,28.2){\line(3,4){10.2}}
\put(212.6,28.2){\line(-3,4){10.2}}
\put(217.4,28.2){\line(3,4){11.2}}

\put(251,70){$A_{41}$}
\multiput(270.35,42)(30,0){2}{$\otimes$}
\multiput(285.35,62)(0,-40){2}{$\otimes$}
\put(277.4,48.2){\line(3,4){10.2}}
\put(302.6,48.2){\line(-3,4){10.2}}
\put(277.4,41.8){\line(3,-4){10.2}}
\put(302.6,41.8){\line(-3,-4){10.2}}
\put(295,63){{\footnotesize $1$}}
\put(280,43){{\footnotesize $k$}}
\put(310,43){{\footnotesize $k$}}
\put(295,23){{\footnotesize $1$}}
\put(251,3){{\footnotesize $d_0 = k+1$}}
\put(310,3){{\footnotesize $f = 1$}}

\put(341,70){$A_{42}$}
\multiput(360.35,22)(30,0){2}{$\otimes$}
\multiput(375.35,42)(0,20){2}{$\otimes$}
\put(380,48.8){\line(0,1){12.5}}
\put(367.4,28.2){\line(3,4){10.2}}
\put(392.6,28.2){\line(-3,4){10.2}}
\put(385,63){{\footnotesize $1$}}
\put(385,43){{\footnotesize $1$}}
\put(370,23){{\footnotesize $k$}}
\put(400,23){{\footnotesize $k$}}
\put(341,3){{\footnotesize $d_0 = k+1$}}
\put(405,3){{\footnotesize $f = 1$}}

\end{picture}

\section{Sincere one parametric 2-equipped posets containing $K_8$, but critical and special sets}\label{sincereK8posets}
\begin{picture}(430,270)


\put(0,0){\line(1,0){430}}
\put(0,90){\line(1,0){430}}
\put(0,180){\line(1,0){430}}
\put(0,270){\line(1,0){430}}

\put(120,180){\line(0,1){90}}
\put(250,180){\line(0,1){90}}

\put(0,0){\line(0,1){270}}
\put(110,0){\line(0,1){90}}
\put(100,90){\line(0,1){90}}
\put(210,0){\line(0,1){180}}
\put(330,0){\line(0,1){90}}
\put(320,90){\line(0,1){90}}
\put(430,0){\line(0,1){270}}

\put(1,260){$A_{26}$}
\multiput(22,212)(30,0){2}{$\circ$}
\put(82,232){$\circ$}
\put(80.35,212){$\otimes$}
\put(85,218.8){\line(0,1){14.3}}

\put(131,260){$A_{27}$}
\multiput(152,212)(30,0){2}{$\circ$}
\put(167,232){$\circ$}
\put(210.35,212){$\otimes$}
\put(156.41,216.88){\line(3,4){12.3}}
\put(183.59,216.88){\line(-3,4){12.3}}

\put(251,260){$A_{30}$}
\multiput(299.5,202)(30,0){2}{$\circ$}
\multiput(362,222)(0,20){2}{$\circ$}
\put(360.35,202){$\otimes$}
\put(365,226.9){\line(0,1){16.2}}
\put(365,208.8){\line(0,1){14.3}}
\put(370,243){{\footnotesize $1$}}
\put(370,223){{\footnotesize $1$}}
\put(370,203){{\footnotesize $2k-2$}}
\put(337.5,203){{\footnotesize $k$}}
\put(307.5,203){{\footnotesize $k$}}
\put(251,183){{\footnotesize $d_0 = 2k$}}
\put(395,183){{\footnotesize $f = 2$}}

\put(1,170){$A_{31}$}
\put(24.5,142){$\circ$}
\multiput(9.5,122)(30,0){2}{$\circ$}
\put(72,142){$\circ$}
\put(70.35,122){$\otimes$}
\put(13.91,126.88){\line(3,4){12.3}}
\put(41.09,126.88){\line(-3,4){12.3}}
\put(75,128.8){\line(0,1){14.3}}
\put(80,143){{\footnotesize $1$}}
\put(80,123){{\footnotesize $2k$}}
\put(32.5,143){{\footnotesize $1$}}
\put(47.5,123){{\footnotesize $k$}}
\put(17.5,123){{\footnotesize $k$}}
\put(1,93){{\footnotesize $d_0 = 2k+1$}}
\put(70,93){{\footnotesize $f = 2$}}

\put(101,170){$A_{32}$}
\multiput(124.5,132)(0,20){2}{$\circ$}
\multiput(109.5,112)(30,0){2}{$\circ$}
\put(170.35,112){$\otimes$}
\put(113.91,116.88){\line(3,4){12.3}}
\put(141.09,116.88){\line(-3,4){12.3}}
\put(127.5,136.9){\line(0,1){16.2}}
\put(180,113){{\footnotesize $2k+2$}}
\put(132.5,153){{\footnotesize $1$}}
\put(132.5,133){{\footnotesize $1$}}
\put(147.5,113){{\footnotesize $k$}}
\put(117.5,113){{\footnotesize $k$}}
\put(101,93){{\footnotesize $d_0 = 2k+2$}}
\put(175,93){{\footnotesize $f = 2$}}

\put(221,170){$A_{35}$}
\multiput(229.5,132)(30,0){2}{$\circ$}
\put(244.5,112){$\circ$}
\put(292,152){$\circ$}
\put(290.35,132){$\otimes$}
\put(233.91,133.12){\line(3,-4){12.3}}
\put(261.08,133.12){\line(-3,-4){12.3}}
\put(295,138.8){\line(0,1){14.3}}
\put(300,153){{\footnotesize $1$}}
\put(300,133){{\footnotesize $2k$}}
\put(267.5,133){{\footnotesize $k$}}
\put(237.5,133){{\footnotesize $k$}}
\put(252.5,113){{\footnotesize $1$}}
\put(221,93){{\footnotesize $d_0 = 2k+1$}}
\put(290,93){{\footnotesize $f = 2$}}

\put(321,170){$A_{36}$}
\multiput(332,132)(30,0){2}{$\circ$}
\put(392,152){$\circ$}
\put(390.35,132){$\otimes$}
\put(392,112){$\circ$}
\put(395,138.8){\line(0,1){14.3}}
\put(395,131.2){\line(0,-1){14.3}}
\put(400,153){{\footnotesize $1$}}
\put(400,133){{\footnotesize $2k-2$}}
\put(400,113){{\footnotesize $1$}}
\put(367.5,133){{\footnotesize $k$}}
\put(337.5,133){{\footnotesize $k$}}
\put(321,93){{\footnotesize $d_0 = 2k$}}
\put(400,93){{\footnotesize $f = 2$}}

\put(1,80){$A_{37}$}
\multiput(24.5,62)(0,-40){2}{$\circ$}
\multiput(9.5,42)(30,0){2}{$\circ$}
\put(60.35,42){$\otimes$}
\put(13.91,46.88){\line(3,4){12.3}}
\put(41.09,46.88){\line(-3,4){12.3}}
\put(13.91,43.12){\line(3,-4){12.3}}
\put(41.09,43.12){\line(-3,-4){12.3}}
\put(70,43){{\footnotesize $2k+2$}}
\put(32.5,63){{\footnotesize $1$}}
\put(47.5,43){{\footnotesize $k$}}
\put(17.5,43){{\footnotesize $k$}}
\put(32.5,23){{\footnotesize $1$}}
\put(1,3){{\footnotesize $d_0 = 2k+2$}}
\put(80,3){{\footnotesize $f = 2$}}

\put(111,80){$A_{40}$}
\put(185.35,32){$\otimes$}
\multiput(127,32)(30,0){2}{$\circ$}
\put(170.35,52){$\otimes$}
\put(161.41,36.88){\line(3,4){11.2}}
\put(187.6,38.2){\line(-3,4){10.2}}

\put(211,80){$A_{43}$}
\put(290.35,42){$\otimes$}
\multiput(222,42)(40,0){2}{$\circ$}
\multiput(275.35,62)(0,-40){2}{$\otimes$}
\put(266.41,46.88){\line(3,4){11.2}}
\put(292.6,48.2){\line(-3,4){10.2}}
\put(266.41,43.12){\line(3,-4){11.2}}
\put(292.6,41.8){\line(-3,-4){10.2}}
\put(285,63){{\footnotesize $1$}}
\put(300,43){{\footnotesize $2k+1$}}
\put(270,43){{\footnotesize $k$}}
\put(285,23){{\footnotesize $1$}}
\put(230,43){{\footnotesize $k+1$}}
\put(211,3){{\footnotesize $d_0 = 2k+2$}}
\put(300,3){{\footnotesize $f = 1$}}

\put(331,80){$A_{44}$}
\multiput(342,42)(60,0){2}{$\circ$}
\multiput(355.35,22)(15,20){3}{$\otimes$}
\multiput(362.4,28.2)(15,20){2}{\line(3,4){10.2}}
\put(392.6,61.8){\line(3,-4){11.2}}
\put(357.4,28.2){\line(-3,4){11.2}}
\put(395,63){{\footnotesize $1$}}
\put(410,43){{\footnotesize $k$}}
\put(380,43){{\footnotesize $2k$}}
\put(350,43){{\footnotesize $k$}}
\put(365,23){{\footnotesize $1$}}
\put(331,3){{\footnotesize $d_0 = 2k+1$}}
\put(400,3){{\footnotesize $f = 1$}}

\end{picture}

\section{Special one parametric equipped posets with their special vectors}\label{specialposets}
\begin{picture}(430,90)
\put(0,90){\line(1,0){430}}
\put(0,0){\line(1,0){430}}

\put(0,0){\line(0,1){90}}
\put(110,0){\line(0,1){90}}
\put(205,0){\line(0,1){90}}
\put(320,0){\line(0,1){90}}
\put(430,0){\line(0,1){90}}
\put(1,80){$A_{45}$}
\multiput(20.35,42)(50,0){2}{$\otimes$}
\multiput(22,62)(50,-40){2}{$\circ$}
\put(25,48.8){\line(0,1){14.3}}
\put(75,41.2){\line(0,-1){14.3}}
\put(26.75,63.6){\line(5,-4){46.6}}
\put(30,63){{\footnotesize $1$}}
\put(30,43){{\footnotesize $2k$}}
\put(80,43){{\footnotesize $2k$}}
\put(80,23){{\footnotesize $1$}}
\put(1,3){{\footnotesize $d_0 = 2k+1$}}
\put(80,3){{\footnotesize $f = 4$}}

\put(111,80){$A_{46}$}
\multiput(140.35,42)(30,0){2}{$\otimes$}
\multiput(155.35,62)(0,-40){2}{$\otimes$}
\put(160,28.8){\line(0,1){32.5}}
\multiput(147.4,48.2)(15,-20){2}{\line(3,4){10.2}}
\multiput(172.6,48.2)(-15,-20){2}{\line(-3,4){10.2}}
\put(165,63){{\footnotesize $1$}}
\put(150,43){{\footnotesize $k$}}
\put(180,43){{\footnotesize $k$}}
\put(165,23){{\footnotesize $1$}}
\put(111,3){{\footnotesize $d_0 = k+1$}}
\put(175,3){{\footnotesize $f = 2$}}

\put(206,80){$A_{47}$}
\put(290.35,42){$\otimes$}
\multiput(217,42)(30,0){2}{$\circ$}
\put(292,22){$\circ$}
\put(232,62){$\circ$}
\put(295,26.9){\line(0,1){14.3}}
\put(221.41,46.88){\line(3,4){12.3}}
\put(248.59,46.88){\line(-3,4){12.3}}
\put(236.8,63.8){\line(3,-2){56.4}}
\put(300,43){{\footnotesize $2k$}}
\put(300,23){{\footnotesize $1$}}
\put(240,63){{\footnotesize $1$}}
\put(253,43){{\footnotesize $k$}}
\put(225,43){{\footnotesize $k$}}
\put(206,3){{\footnotesize $d_0 = 2k+1$}}
\put(290,3){{\footnotesize $f = 4$}}

\put(321,80){$A_{48}$}
\multiput(342,42)(60,0){2}{$\circ$}
\multiput(355.35,22)(15,20){3}{$\otimes$}
\qbezier(362.4,28.2)(360,60)(387.6,61.8)
\multiput(362.4,28.2)(15,20){2}{\line(3,4){10.2}}
\put(392.6,61.8){\line(3,-4){11.2}}
\put(357.4,28.2){\line(-3,4){11.2}}
\put(395,63){{\footnotesize $1$}}
\put(410,43){{\footnotesize $k$}}
\put(380,43){{\footnotesize $2k$}}
\put(350,43){{\footnotesize $k$}}
\put(365,23){{\footnotesize $1$}}
\put(321,3){{\footnotesize $d_0 = 2k+1$}}
\put(400,3){{\footnotesize $f = 2$}}
\end{picture}

\section{Dimensions of $A_{25}, A_{26}, A_{27}, A_{38}, A_{39}, A_{40}$ and their dual sets}\label{dimensions}
\begin{center}
\begin{tabular}{|l|c|c|c|l|c|c|c|}\h
\multicolumn{4}{|c|}{\Mtres} &   \multicolumn{4}{c|}{\Mocho}  \\ \h
\multicolumn{4}{|c|}{$\dimc = (d_0; d_a, d_b, d_{\eta} )$} & \multicolumn{4}{c|}{$\dimc = (d_0; d_a, d_b, d_q )$} \\ \h
T & f & $\dimc A_{25}$ & $\dimc A_{25}^*$  & T & f & $\dimc A_{38}$ & $\dimc A_{38}^*$  \\ \h
1 & 1 & (1,1,0,1) & (1,1,0,1) & 1 & 1 & (1,0,0,1) & (1,1,1,1) \\
2 & 2 & (2,2,0,1) & (2,2,2,1) & 2 & 1 & (1,0,1,1) & (1,1,0,1) \\
3 & 2 & (1,0,0,1) & (1,2,0,1) & 3 & 1 & (1,1,1,1) & (1,0,0,1) \\
4 & 2 & (1,2,0,1) & (1,0,0,1) & 4 & 2 & (1,0,0,2) & (1,0,0,2) \\
5 & 2 & (2,2,2,1) & (2,2,0,1) &   &   &           &           \\ \h 
\end{tabular}\end{center}
\begin{center}
\begin{tabular}{|l|c|c|c|l|c|c|c|}\h
\multicolumn{4}{|c|}{\Ltres} & \multicolumn{4}{c|}{\Lcuatro}    \\ \h
\multicolumn{4}{|c|}{$\dimc = (d_0; d_{\varrho}, d_{\sigma}, d_{a}, d_{\eta})$} & \multicolumn{4}{c|}{$\dimc = (d_0; d_{\varrho}, d_{\sigma}, d_{a}, d_{\zeta})$} \\ \h
T & f & $\dimc A_{26}$ & $\dimc A_{26}^*$ & T & f & $\dimc A_{27}$ & $\dimc A_{27}^*$  \\ \h
1 & 1 & (2,1,1,1,1) & (2,1,1,1,1) & 1 & 1 & (1,0,0,1,1) & (1,0,0,1,1) \\
2 & 2 & (1,0,1,0,1) & (1,1,0,0,1) & 2 & 2 & (1,0,0,0,1) & (1,0,0,2,1) \\
3 & 2 & (2,1,1,2,1) & (2,1,1,0,1) & 3 & 2 & (2,0,1,2,1) & (2,0,1,2,1) \\
4 & 2 & (1,0,0,0,1) & (1,1,1,0,1) & 4 & 2 & (1,0,0,2,1) & (1,0,0,2,1) \\
5 & 2 & (1,1,1,0,1) & (1,0,0,0,1) & 5 & 2 & (2,1,1,2,1) & (2,0,0,2,1) \\
6 & 2 & (2,1,1,0,1) & (2,1,1,2,1) & 6 & 2 & (2,0,0,2,1) & (2,1,1,2,1) \\ \h
\end{tabular}
\end{center}

\begin{center}
\begin{tabular}{|l|c|c|c||l|c|c|c|}\h
\multicolumn{8}{|c|}{\Mnueve} \\ \h
\multicolumn{8}{|c|}{$\dimc = (d_0; d_a, d_b, d_q, d_{\theta} )$} \\ \h
T  & f & $\dimc A_{39}$ & $\dimc A_{39}^*$ & T  & f & $\dimc A_{39}$ & $\dimc A_{39}^*$\\ \h
1  & 1 & (2,1,0,1,1) & (2,1,0,2,1) & 6  & 2 & (1,0,0,2,1) & (3,2,0,2,1)\\
2  & 1 & (1,1,0,1,1) & (2,1,0,1,1) & 7  & 2 & (2,2,0,2,1) & (2,0,0,2,1)\\
3  & 1 & (1,0,0,1,1) & (2,2,0,1,1) & 8  & 2 & (2,0,0,2,1) & (2,2,0,2,1)\\
4  & 1 & (2,2,0,1,1) & (1,0,0,1,1) & 9  & 2 & (3,2,0,2,2) & (1,0,0,2,1)\\
5  & 1 & (2,1,0,2,1) & (1,1,0,1,1) & 10 & 2 & (3,2,0,2,1) & (3,2,0,2,2)\\ \h
\end{tabular}\end{center}
\begin{center}
\begin{tabular}{|l|c|c|c||l|c|c|c|} \h
\multicolumn{8}{|c|}{\Lonce}    \\ \h
\multicolumn{4}{|c||}{$\dimc = (d_0; d_{\varrho}, d_{\sigma}, d_{a}, d_q)$} & \multicolumn{4}{c|}{$\dimc = (d_0; d_{\varrho}, d_{\sigma}, d_{a}, d_q)$} \\ \h
T & f & $\dimc A_{40}$ & $\dimc A_{40}^*$ & T & f & $\dimc A_{40}$ & $\dimc A_{40}^*$  \\ \h
1 & 1 & (1,1,0,0,1) & (1,0,0,1,1) & 9   & 2 & (1,0,0,0,2) & (3,2,0,2,2) \\
2 & 1 & (2,1,1,1,1) & (2,1,0,2,1) & 10  & 2 & (2,1,0,2,2) & (2,1,0,0,2) \\
3 & 1 & (2,1,1,2,1) & (2,1,0,1,1) & 11  & 1 & (2,1,0,1,2) & (2,1,1,2,1) \\
4 & 2 & (1,1,0,0,2) & (3,1,0,2,1) & 12  & 1 & (2,1,0,1,1) & (2,1,0,1,2) \\
5 & 1 & (1,1,0,1,1) & (1,1,1,0,1) & 13  & 1 & (1,0,0,0,1) & (1,1,0,1,1) \\
6 & 1 & (2,1,0,2,1) & (2,1,1,1,1) & 14  & 2 & (2,1,0,0,2) & (2,1,0,2,2) \\
7 & 1 & (1,0,0,1,1) & (1,1,0,0,1) & 15  & 2 & (3,1,0,2,2) & (1,1,0,0,2) \\ 
8 & 2 & (3,2,0,2,2) & (1,0,0,0,2) &     &   &             &              \\ \h 
\end{tabular}
\end{center}

\section{Dimension vectors of $K_7$ and $K_9$}
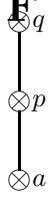
\label{dimensions of K7 and K9}

\begin{picture}(150,70)
\multiput(0.35,22)(0,30){3}{$\otimes$}
\multiput(5,28.8)(0,30){2}{\line(0,1){22.3}}
\put(32,22){$\circ$}
\put(10,83){{\footnotesize $q$}}
\put(10,53){{\footnotesize $p$}}
\put(10,23){{\footnotesize  $a$}}
\put(40,23){{\footnotesize  $\theta$}}
\put(15,0){$K_7$}
\put(100,23){$d = (d_0, d_a, d_p, d_q, d_{\theta})$ y tamaño de paso $\mu = (2,1,1,1,1)$.}
\end{picture}

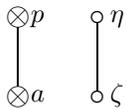
\scalebox{0.85}{
\renewcommand{\arraystretch}{0.92}
\renewcommand{\tabcolsep}{1.5mm}
\setlength{\doublerulesep}{0.14pt}
\setlength{\arrayrulewidth}{0.2pt}
\begin{tabular}{|l|l|p{0.3cm}|l|l|p{0.6cm}|l|l|p{0.3cm}|l|l|}
\multicolumn{1}{l}{f = 1} &\mcf & \mcf & \multicolumn{1}{l}{f = 1}& \mcf & \mcf &\multicolumn{1}{l}{f = 2} & \mcf & \mcf &\multicolumn{1}{l}{f = 2} & \mcf \\ \cline{1-2}\cline{4-5}\cline{7-8}\cline{10-11}
1   & (2,1,1,2,1)&    &$1^* $& (2,0,2,1,1)&    &13 & (3,2,2,2,1)&    &$13^* $& (3,0,2,2,2)\\
2   & (1,1,1,0,0)&    &$2^* $& (1,0,0,1,1)&    &14 & (1,0,2,0,0)&    &$14^* $& (1,0,0,2,1)\\
3   & (2,0,1,2,1)&    &$3^* $& (2,1,2,1,1)&    &15 & (1,0,0,2,0)&    &$15^* $& (1,0,2,0,1)\\
4   & (2,1,0,2,1)&    &$4^* $& (2,1,2,0,1)&    &16 & (2,0,2,0,1)&    &$16^* $& (1,2,0,0,0)\\
5   & (1,0,1,1,0)&    &$5^* $& (1,0,1,1,1)&    &17 & (1,0,0,0,1)&    &$17^* $& (2,2,0,2,1)\\
6   & (1,1,0,1,0)&    &$6^* $& (1,0,1,0,1)&    &18 & (3,2,2,0,1)&    &$18^* $& (3,2,0,2,2)\\
7   & (2,1,1,0,1)&    &$7^* $& (2,2,0,1,1)&    &19 & (3,2,0,2,1)&    &$19^* $& (3,2,2,0,2)\\
8   & (2,0,1,1,1)&    &$8^* $& (2,2,1,1,1)&    &20 & (4,2,2,2,1)&    &$20^* $& (4,2,2,2,3)\\
9   & (1,1,0,0,0)&    &$9^* $& (1,1,0,0,1)&    &21 & (2,0,0,2,1)&    &$21^* $& (2,2,2,0,1)\\
10  & (2,1,0,1,1)&    &$10^*$& (2,2,1,0,1)&    &22 & (1,0,0,0,0)&    &$22^*$ & (1,2,0,0,1)\\
11  & (1,0,1,0,0)&    &$11^*$& (1,1,0,1,1)&    &23 & (3,0,2,2,1)&    &$23^*$ & (3,2,2,2,2)\\ \cline{10-11}
12  & (1,0,0,1,0)&    &$12^*$& (1,1,1,0,1)&    &24 & (2,0,2,2,1)&\mcf&  \mcf &  \mcf      \\ \cline{1-2}\cline{4-5}
\mcf&   \mcf     &\mcf& \mcf &  \mcf      &    &25 & (2,2,0,0,1)&\mcf&  \mcf &   \mcf     \\ \cline{7-8}
\end{tabular}}

\begin{picture}(50,60)
\multiput(0.35,22)(0,30){2}{$\otimes$}
\put(5,28.8){\line(0,1){22.3}}
\multiput(32,22)(0,30){2}{$\circ$}
\put(35,27){\line(0,1){25.9}}
\put(10,53){{\footnotesize $p$}}
\put(10,23){{\footnotesize  $a$}}
\put(40,53){{\footnotesize  $\eta$}}
\put(40,23){{\footnotesize  $\zeta$}}
\put(15,0){$K_9$}
\put(100,23){$d = (d_0, d_a, d_p, d_{\zeta}, d_{\eta})$ y tamaño de paso $\mu = (3,2,2,1,1)$.}
\end{picture}

\scalebox{0.85}{
\renewcommand{\arraystretch}{0.92}
\renewcommand{\tabcolsep}{1.5mm}
\setlength{\doublerulesep}{0.14pt}
\setlength{\arrayrulewidth}{0.2pt}
\begin{tabular}{|l|l|p{0.3cm}|l|l|p{0.6cm}|l|l|p{0.3cm}|l|l|}
\multicolumn{1}{l}{f = 1} &\mcf & \mcf & \multicolumn{1}{l}{f = 1}& \mcf & \mcf &\multicolumn{1}{l}{f = 2} & \mcf & \mcf &\multicolumn{1}{l}{f = 2} & \mcf \\ \cline{1-2}\cline{4-5}\cline{7-8}\cline{10-11}
1 & (3,2,3,1,1)& &$1^* $& (3,1,3,1,1)& &14 & (2,2,2,0,1)& &$14^* $& (2,0,2,1,1)\\
2 & (3,2,1,1,1)& &$2^* $& (3,3,1,1,1)& &15 & (1,0,0,1,0)& &$15^* $& (1,2,0,0,0)\\
3 & (2,2,1,0,1)& &$3^* $& (2,1,1,1,1)& &16 & (1,0,2,1,0)& &$16^* $& (1,0,2,0,0)\\
4 & (1,0,1,1,0)& &$4^* $& (1,1,1,0,0)& &17 & (2,2,0,0,1)& &$17^* $& (2,2,0,1,1)\\
5 & (2,2,1,1,0)& &$5^* $& (2,1,1,1,0)& &18 & (2,2,2,1,0)& &$18^* $& (2,0,2,1,0)\\
6 & (1,0,1,0,1)& &$6^* $& (1,1,1,0,1)& &19 & (1,0,0,0,1)& &$19^* $& (1,2,0,0,1)\\
7 & (1,1,0,0,0)& &$7^* $& (1,1,0,1,0)& &20 & (3,2,2,2,0)& &$20^* $& (3,2,2,1,0)\\
8 & (2,1,2,0,1)& &$8^* $& (2,1,2,1,1)& &21 & (3,2,2,0,2)& &$21^* $& (3,2,2,1,2)\\
9 & (1,0,1,0,0)& &$9^* $& (1,1,1,1,0)& &22 & (3,2,2,0,1)& &$22^* $& (3,2,2,2,1)\\
10& (3,1,2,1,1)& &$10^*$& (3,3,2,1,1)& &23 & (1,0,0,0,0)& &$23^*$ & (1,2,0,1,0)\\
11& (2,1,1,0,1)& &$11^*$& (2,2,1,1,1)& &24 & (2,0,2,0,1)& &$24^*$ & (2,2,2,1,1)\\ \cline{4-5}\cline{10-11}
12& (2,1,2,1,0)&\mcf & \mcf &    \mcf    & &25 & (2,2,0,1,0)&\mcf &  \mcf &    \mcf    \\
13& (1,1,0,0,1)&\mcf & \mcf &    \mcf    & &26 & (1,0,2,0,1)&\mcf &  \mcf &    \mcf    \\ \cline{1-2}\cline{7-8}
\end{tabular}}

\end{document}